\documentclass{amsart}
\usepackage{amsmath, amssymb, amsthm}
\usepackage{pst-all}

\numberwithin{equation}{section}

\theoremstyle{plain}
\newtheorem{prop}{Proposition}[section]
\newtheorem{conj}[prop]{Conjecture}

\newtheorem{lemm}[prop]{Lemma}
\newtheorem{ques}[prop]{Question}

\theoremstyle{definition}
\newtheorem{defi}[prop]{Definition}
\newtheorem{rema}[prop]{Remark}
\newtheorem{exam}[prop]{Example}

\psset{unit=1mm}
\psset{dash=1.5pt 1.5pt}
\psset{linewidth=0.8pt}
\psset{arrowsize=3.5pt}
\psset{doublesep=1pt}
\newpsstyle{etc}{linestyle=dotted}
\newpsstyle{exist}{linestyle=dashed}

\newcommand\AD[1]{{\scriptstyle\mathrm{A\hspace{-0.1ex}D}}(#1)}

\newcommand\agoes[1]{\mathrel{\overset{{}_{#1}}{\rightsquigarrow_{\!\!\!\!\!\!-}}}}
\newcommand\agoesR[1]{\mathrel{\overset{{}_{#1}}{\rightsquigarrow_{\!\!\!\!\!\!-\,\smallR}}}}
\newcommand\agoesRlong[1]{\mathrel{\overset{{}_{#1}}{\ \rightsquigarrow_{\!\!\!\!\!\!-\,\smallR}}}}
\newcommand\agozotR{\mathrel{\overset{0,1,2}{\ \rightsquigarrow_{\!\!\!\!\!\!-\,\smallR}}}}
\newcommand\agozotiR{\mathrel{\overset{0,1,2,\infty}{\ \rightsquigarrow_{\!\!\!\!\!\!-\,\smallR}}}}

\newcommand\AL[1]{\ell(#1)}
\newcommand\AN[1]{a(#1)}
\newcommand\ALL[2]{\ell_{#1}(#2)}
\newcommand\ar[1]{\rr_{[#1]}}
\newcommand\as[1]{\ss_{[#1]}}
\newcommand\aS{\underline{S}}

\newcommand\at[1]{\tt_{[#1]}}
\newcommand\atta[1]{\tta_{#1}}
\newcommand\attA[1]{\ttA_{#1}}
\newcommand\attb[1]{\ttb_{#1}}
\newcommand\attB[1]{\ttB_{#1}}
\newcommand\attc[1]{\ttc_{#1}}
\newcommand\attC[1]{\ttC_{#1}}

\newcommand\CC{C}
\newcommand\cl[1]{\overline{\vrule width0pt height5pt#1}}

\newcommand\dd{d}
\newcommand\DL[1]{{\scriptstyle\mathrm{D}}\smalll(#1)}
\newcommand\DR[1]{{\scriptstyle\mathrm{D}}\smallr(#1)}

\newcommand\ee{e}

\newcommand\eqR{\equiv_\smallR}
\newcommand\eqpR{\equiv_\smallR^{\scriptscriptstyle+}}
\newcommand\ew{\varepsilon}

\newcommand\ff{f}

\newcommand\ffs{\ff_{\hspace{-0.3ex}*\hspace{0.2ex}}}

\let\ge=\geqslant
\renewcommand\gg{g}
\newcommand\GG{G}
\newcommand\GGp{G^{\scriptscriptstyle+}}
\newcommand\ggs{\gg_{\hspace{-0.2ex}*\hspace{0.2ex}}}

\newcommand\goes[1]{\overset{{}_{#1}}{\rightsquigarrow}}
\newcommand\goesR[1]{\overset{{}_{#1}}{\rightsquigarrow}_{\hspace{-0.4ex}{\scriptscriptstyle R}}}

\newcommand\gozD{\mathrel{\overset{\scriptscriptstyle 0,D}{\rightsquigarrow}}}
\newcommand\gozDR{\mathrel{\gozD_{\hspace{-0.6ex}{\scriptscriptstyle R}}}}
\newcommand\gozDRh{\mathrel{\gozD_{\widehat\smallR}}}

\newcommand\gooR{\mathrel{\overset{1}{\ \rightsquigarrow_\smallR}}}
\newcommand\gotR{\mathrel{\overset{2}{\ \rightsquigarrow_\smallR}}}

\newcommand\gozoi{\mathrel{\overset{\scriptscriptstyle 0,1,\infty}{\rightsquigarrow}}}
\newcommand\gozoiR{\mathrel{\gozoi_{\hspace{-1.2ex}{\scriptscriptstyle R}}}}

\newcommand\gozo{\mathrel{\overset{\scriptscriptstyle 0,1}{\rightsquigarrow}}}
\newcommand\gozoR{\mathrel{\gozo_{\hspace{-0.4ex}{\scriptscriptstyle R}}}}

\newcommand\gozot{\mathrel{\overset{\scriptscriptstyle 0,1,2}{\rightsquigarrow}}}
\newcommand\gozotR{\mathrel{\gozot_{\hspace{-0.9ex}{\scriptscriptstyle R}}}}
\newcommand\gozotRe{\mathrel{\gozot_{\hspace{-0.9ex}{\scriptscriptstyle R}_{\hspace{-0.2ex}e}}}}
\newcommand\gozotRo{\mathrel{\gozot_{\hspace{-0.9ex}{\scriptscriptstyle R}_{\hspace{-0.2ex}1}}}}
\newcommand\gozotRt{\mathrel{\gozot_{\hspace{-0.9ex}{\scriptscriptstyle R}_{\hspace{-0.2ex}2}}}}

\newcommand\gozoti{\mathrel{\overset{\scriptscriptstyle 0,1,2,\infty}{\rightsquigarrow}}}
\newcommand\gozotiR{\mathrel{\gozoti_{\hspace{-1.8ex}{\scriptscriptstyle R}\hspace{0.5ex}}}}

\newcommand\gozt{\mathrel{\overset{\scriptscriptstyle 0,2}{\rightsquigarrow}}}
\newcommand\goztR{\gozt_{\hspace{-0.4ex}{\scriptscriptstyle R}}}

\newcommand\GR[2]{\langle#1{\mid}\nobreak#2\rangle}

\newcommand\hh{h}

\newcommand\HR[1]{{\scriptstyle\mathrm{H\hspace{-0.1ex}R}}(#1)}

\newcommand\ie{e}
\newcommand\ii{i}
\newcommand\ince{\subseteq}
\newcommand\inv{^{-1}}
\newcounter{ITEM}
\newcommand\ITEM[1]{\setcounter{ITEM}{#1}\leavevmode\hbox{\rm(\roman{ITEM})}}

\newcommand\jj{j}

\newcommand\kk{k}

\let\le=\leqslant
\newcommand\LG[1]{\vert#1\vert}%Length of a word

\newcommand\mm{m}
\newcommand\MM{M}
\newcommand\MON[2]{\langle#1{\mid}\nobreak#2\rangle^{\!\scriptscriptstyle+}}

\newcommand\NL[1]{{\scriptstyle\mathrm{N}}\smalll(#1)}
\newcommand\NR[1]{{\scriptstyle\mathrm{N}}\smallr(#1)}

\newcommand\oo[1]{o(#1)}%origin

\newcommand\pdots{\hspace{0.2ex}{\cdot}{\cdot}{\cdot}\hspace{0.5ex}}
\newcommand\pp{p}
\newcommand\PPP{\mathcal{P}}

\newcommand\PPPPe[3]{\mathcal{P}_{\!e}(#1, #2, #3)}

\newcommand\PPPPo[3]{\mathcal{P}_{1}(#1, #2, #3)}

\newcommand\Prop{\mathrm{H}}
\newcommand\Propp{\Prop^\sharp}

\newcommand\qq{q}
\newcommand\quot{{/}\hspace{-0.7ex}}

\newcommand\resp{{\it resp}}
\newcommand\rr{r}
\newcommand\RR{R}
\newcommand\RRh{\widehat\RR}

\newcommand\SG[1]{\langle\!\langle{#1}\rangle\!\rangle}
\newcommand\SGp[1]{\langle\!\langle{#1}\rangle\!\rangle^{\!\scriptscriptstyle+}}
\newcommand\smalll{_{\ell}}
\newcommand\smallr{_{r}}
\newcommand\smallR{{\!{}_\RR}}
\renewcommand\ss{s}
\renewcommand\SS{S}

\newcommand\tail[1]{\tau(#1)}
\renewcommand\tt{t}

\newcommand\tta{\mathtt{a}}
\newcommand\ttA{\mathtt{A}}
\newcommand\ttb{\mathtt{b}}
\newcommand\ttB{\mathtt{B}}
\newcommand\ttc{\mathtt{c}}
\newcommand\ttC{\mathtt{C}}
\newcommand\ttd{\mathtt{d}}
\newcommand\ttD{\mathtt{D}}
\newcommand\tte{\mathtt{e}}
\newcommand\ttf{\mathtt{f}}
\newcommand\TTT{\boldsymbol{T}}
\newcommand\tw{\widehat\ww}

\newcommand\uu{u}

\newcommand\vv{v}
\newcommand\vvs{\vv_{\hspace{-0.1ex}*\hspace{0.2ex}}}

\newcommand\wdots{, ...\hspace{0.2ex},}
\newcommand\ww{w}
\newcommand\wws{\ww_{\hspace{-0.1ex}*\hspace{0.2ex}}}
\newcommand\WWW[1]{\mathcal{W}(#1)}
\newcommand\WWWp[1]{\mathcal{W}^{\scriptscriptstyle+}\!(#1)}

\begin{document}

\author{Patrick DEHORNOY}

\address{Laboratoire de Math\'ematiques Nicolas Oresme, UMR 6139, UCBN et CNRS, 14032 Caen, France}
\email{patrick.dehornoy@unicaen.fr}
\urladdr{//www.math.unicaen.fr/\!\hbox{$\sim$}dehornoy}

\author{Eddy GODELLE}

\address{Laboratoire de Math\'ematiques Nicolas Oresme, UMR 6139, UCBN et CNRS, 14032 Caen, France}
\email{eddy.godelle@unicaen.fr}
\urladdr{//www.math.unicaen.fr/\!\hbox{$\sim$}godelle}

\title{A conjecture about Artin--Tits groups}



\begin{abstract}
We conjecture that the word problem of Artin--Tits groups can be solved without introducing trivial factors $\ss\ss\inv $ or $\ss\inv \ss$. Here we make this statement precise and explain how it can be seen as a weak form of hyperbolicity. We prove the conjecture in the case of Artin--Tits groups of type~FC, and we discuss various possible approaches for further extensions, in particular a syntactic argument that works at least in the right-angled case.
\end{abstract}

\thanks{Work partially supported by the ANR grant TheoGar ANR-08-BLAN-0269-02}

\maketitle

Artin--Tits groups, also known as Artin groups, are those groups that are defined by relations of the form
\begin{equation}
\label{E:Relation}
\tag{1}
\ss \tt \ss ... = \tt \ss \tt ...
\end{equation}
where both terms consist of two alternating letters and have the same length. These groups were first investigated by J.~Tits in the late 1960's \cite{Bri}. Two seminal references are~\cite{BrS} and~\cite{Dlg}. At the moment many questions involving these groups remain open, including the decidability of the word problem in the general case~\cite{Cha}. 

The aim of this paper is to discuss a conjecture about the word problem of Artin--Tits groups, namely that, when a word~$\ww$ represents~$1$, hence is equivalent to the empty word modulo the relations, then $\ww$ can be transformed to the empty word without introducing trivial factors~$\ss \ss\inv$ or~$\ss\inv \ss$ at any intermediate step of the derivation. Here we shall elaborate this intuitive idea into a precise statement, called Property~$\Prop$, and interpret the latter as a weak form of hyperbolicity (whence the notation). So, our conjecture is that all Artin--Tits groups---or, rather, all Artin--Tits presentations since Property~$\Prop$ involves a specific presentation---satisfy Property~$\Prop$. It is known that Artin--Tits groups are not hyperbolic in general and, therefore, our conjecture is about how much of hyperbolicity could nevertheless be true in these groups. The main positive result we prove here is that all Artin--Tits presentation of type~FC satisfy Property~$\Prop$. In terms of applications, Property~$\Prop$ does not imply the decidability of the word problem, at least directly, but it implies that the monoid defined by a family of relations embeds in the corresponding group. Such a result is known to be true for Artin--Tits groups~\cite{Par}, but, so far, the only available proof is quite indirect and relies on the existence of linear representations. If our conjecture is true, it would in particular lead to a more direct proof of that result.

The article is organized as follows. In Section~\ref{S:Prop}, we introduce Property~$\Prop$, connect it with hyperbolicity, and show as a warm-up exercise that all Artin--Tits presentations of spherical type satisfy Property~$\Prop$. We also mention the application to monoid embeddability. Next, in Section~\ref{S:FC}, we establish that all Artin--Tits presentations of type~FC satisfy Property~$\Prop$. Finally, in Section~\ref{S:RAAG}, we give an alternative proof of the result that all right-angled Artin--Tits presentations (a subfamily of type~FC) satisfy Property~$\Prop$. The current range of the argument is more limited than that of Section~\ref{S:FC}, but it might be more promising in terms of subsequent developments as it involves no normal form. 

%%%%

\section{Property~$\Prop$}
\label{S:Prop}

We introduce special transformations and Property~$\Prop$ in Subsection~\ref{SS:Special}. We discuss its connection with the Dehn algorithm in Subsection~\ref{SS:Dehn}. Then, in Subsection~\ref{SS:Spherical}, we establish that all Artin--Tits presentations of spherical type satisfy Property~$\Prop$. Finally, we mention in Subsection~\ref{SS:Embed} an application of Property~$\Prop$ to the embeddability of a monoid~$\MON\SS\RR$ in the corresponding group~$\GR\SS\RR$.

%%%%%%%%

\subsection{Special transformations}
\label{SS:Special}

For $\SS$ a nonempty set, we denote by~$\WWWp\SS$ the free monoid consisting of all words in the alphabet~$\SS$, and by~$\WWW\SS$ the free monoid of all words in~$\SS \cup \SS\inv $, where $\SS\inv $ denotes a disjoint copy of~$\SS$ containing one letter~$\ss\inv $ for each letter~$\ss$ of~$\SS$. By definition, a group~$\GG$ admits the presentation~$\GR\SS\RR$ if $\GG$ is isomorphic to $\WWW\SS \quot \eqR$, where $\eqR$ is the least congruence on~$\WWW\SS$ containing the relations of~$\RR$ plus the free group relations $\ss\inv \ss = \ss \ss\inv = 1$. For~$\ww$ in~$\WWW\SS$, we denote by~$\cl\ww$ the element of~$\GR\SS\RR$ represented by~$\ww$. We use~$\ew$ for the empty word, and, for~$\ww$ in~$\WWW\SS$, we denote by~$\ww\inv$ the word obtained from~$\ww$ by exchanging the letters~$\ss$ and~$\ss\inv$ everywhere and reversing their order. Finally, we say that a word~$\vv$ is a \emph{factor} of another word~$\ww$ if there exist~$\ww_1, \ww_2$ satisfying $\ww = \ww_1 \vv \ww_2$. 

Hereafter, we shall consider presentations of a special type.

\begin{defi}
We say that a group presentation~$(\SS, \RR)$ is \emph{positive} if all relations of~$\RR$ are of the form $\vv = \vv'$---or, equivalently, $\vv\inv \vv' = 1$---where $\vv, \vv'$ are nonempty words of~$\WWWp\SS$. 
\end{defi}

So the point is that the relations can be written with no letter~$\ss\inv$ and no empty word. As shows~\eqref{E:Relation}, Artin--Tits presentations are positive. If $(\SS, \RR)$ is a positive presentation, it makes sense to consider the monoid presented by~$(\SS, \RR)$, which will be denoted by~$\MON\SS\RR$. In the context of positive presentations, the definition of the group~$\GR\SS\RR$ implies that two words~$\ww, \ww'$ of~$\WWW\SS$ represent the same element of~$\GR\SS\RR$ if and only if one can go from~$\ww$ to~$\ww'$ by a finite sequence of transformations of the following types:

\ITEM1 Remove some pair~$\ss\inv \ss$ or~$\ss \ss\inv $, where $\ss$ is a letter of~$\SS$;

\ITEM2 Replace some factor~$\vv$ by~$\vv'$, where $\vv = \vv'$ is a relation of~$\RR$;

\ITEM3 Insert some pair~$\ss\inv \ss$ or~$\ss \ss\inv $, where $\ss$ is a letter of~$\SS$.

\noindent The question we will investigate is the extent to which transformations of the third type (inserting trivial factors~$\ss\inv \ss$ or~$\ss \ss\inv $) can be avoided. If we consider arbitrary words~$\ww, \ww'$, we cannot expect a positive answer, as removing type~\ITEM3 transformations yields a relation that is not symmetric. However, the question makes sense when $\ww'$ is the empty word~$\ew$:

\begin{ques}
\label{Q:Naive}
Assume that $(\SS, \RR)$ is a positive presentation and $\ww$ is a word of~$\WWW\SS$ that represents~$1$ in~$\GR\SS\RR$. Can one transform~$\ww$ into the empty word by using transformations of type~\ITEM1 and~\ITEM2 only?
\end{ques}

The answer to Question~\ref{Q:Naive} is positive when $\RR$ is empty, that is, in the case of a free group, but the question is not well posed in general. For instance, if $\tta \ttb = \ttb \tta$ belongs to~$\RR$, the word $\tta \ttb\inv \tta\inv \ttb$ represents~$1$, but it is eligible for no transformation of type~\ITEM1 and~\ITEM2. This is due to a lack of symmetry in type~\ITEM2 transformations, and it can be fixed by extending the definition so as to also allow replacing~$\vv\inv $ by~$\vv'{}\inv $ when $\vv = \vv'$ is a relation of~$\RR$. However, even so, the answer to Question~\ref{Q:Naive} remains negative in simple cases. For instance, assume $\SS = \{\tta, \ttb, \ttc\}$ and $\RR = \{\tta \ttb = \ttb \tta, \ttb \ttc = \ttc \ttb, \tta \ttc = \ttc \tta\}$. Let $\ww = \tta \ttB \ttc \ttA \ttb \ttC$, where we adopt the convention that, in concrete examples, $\ttA, \ttB$, ... stand for~$\tta\inv, \ttb\inv$, ... \cite{Eps}. Then $\ww$ represents~$1$ in~$\GR\SS\RR$, 
%as shows for instance the derivation
%\begin{multline*}
%\tta \ttB \ttc \ttA \ttb \ttC
%\goes{} \tta \ttB \ttA \tta \ttc \ttA \ttb \ttC
%\goes{} \tta \ttA \ttB \tta \ttc \ttA \ttb \ttC
%\goes{} \ttB \tta \ttc \ttA \ttb \ttC
%\goes{} \ttB \ttc \tta \ttA \ttb \ttC
%\goes{} \ttB \ttc \ttb \ttC
%\goes{} \ttB \ttb \ttc \ttC
%\goes{} \ttc \ttC
%\goes{} \ew,
%\end{multline*},
but it is eligible for no transformation of type~\ITEM1 or~\ITEM2.

Here comes the point. In addition to the above transformations, we consider a new type of syntactic transformations, namely the \emph{subword reversing} transformations of~\cite{Dgp, Dia}.

\begin{defi}
\label{D:Special}
For $(\SS, \RR)$ a positive presentation, we define \emph{special transformations} on words of~$\WWW\SS$ as follows:\\
- \emph{type $0$} (\resp. $\infty$): Remove (\resp. insert) some pair~$\ss\inv \ss$ or~$\ss \ss\inv $, where $\ss$ is a letter of~$\SS$,\\
- \emph{type $1$}: Replace some factor $\vv$ by~$\vv'$, or $\vv\inv $ by~$\vv'{}\inv $, where $\vv = \vv'$ is a relation of~$\RR$,\\
- \emph{type $2\smallr$} (``type $2$-right''): Replace some factor $\vv\inv \vv'$ with~$\uu \uu'{}\inv $, where $\vv, \vv'$ are nonempty and $\vv \uu = \vv' \uu'$ is a relation~of~$\RR$,\\
- \emph{type $2\smalll$} (``type $2$-left''): Replace some factor $\vv \vv'{}\inv $ by~$\uu\inv \uu'$, where $\vv, \vv'$ are nonempty and $\uu \vv = \uu' \vv'$ is a relation of~$\RR$.\\
For~$X \ince \{0,1, 2, \infty\}$, we write $\ww \goesR{\phantom1X\phantom1} \ww'$ if $\ww$ can be transformed to~$\ww'$ using special transformations of type belonging to~$X$.
\end{defi} 

\begin{figure}
\begin{picture}(110,24)(0,-2)
\put(0,19){type $1$:}
\psbezier[style=exist]{->}(0,8)(8,15)(22,15)(30,8)
\psbezier{->}(0,6)(8,-1)(22,-1)(30,6)
\put(15,-2){$\vv$}
\put(15,15){$\vv'$}
\put(40,19){type $2\smallr$:}
\psbezier{->}(40,8)(43,12)(50,14)(55,14)
\psbezier{->}(40,6)(43,2)(50,0)(55,0)
\psbezier[style=exist]{->}(56,14)(59,14)(67,12)(70,8)
\psbezier[style=exist]{->}(56,0)(59,0)(67,2)(70,6)
\put(44,-0.5){$\vv$}
\put(44,13.5){$\vv'$}
\put(66,13){$\uu'$}
\put(66,0){$\uu$}
\put(80,19){type $2\smalll$:}
\psbezier[style=exist]{->}(80,8)(83,12)(90,14)(95,14)
\psbezier[style=exist]{->}(80,6)(83,2)(90,0)(95,0)
\psbezier{->}(96,14)(99,14)(107,12)(110,8)
\psbezier{->}(96,0)(99,0)(107,2)(110,6)
\put(84,-0.5){$\uu$}
\put(84,13.5){$\uu'$}
\put(106,13){$\vv'$}
\put(106,0){$\vv$}
\end{picture}
\caption[]{\sf\smaller Special transformations viewed in the Cayley graph: relations correspond to oriented cells; special transformations of types $1$ and~$2$ correspond to changing way around a cell: in type~$1$, the old path (plain) follows one side of the cell whereas the new path (dotted) follows the other side; in type~$2\smallr$, the old path crosses the source of the cell, with at least one letter on both sides, and the new path crosses the target instead; type~$2\smalll$ is symmetric; note that a transformation of type~$2\smallr$ is not just the inverse of a transformation of type~$2\smalll$ as the definition requires that the initial words~$\vv, \vv'$ are not empty whereas it requires nothing for the final words~$\uu, \uu'$: for instance, if $\vv = \vv'$ is a relation of~$\RR$, then $\vv\inv \vv' \goesR{2\smallr} \ew$ holds, but $\ew \goesR{2\smalll} \vv\inv \vv'$ does not.}
\label{F:Special}
\end{figure}
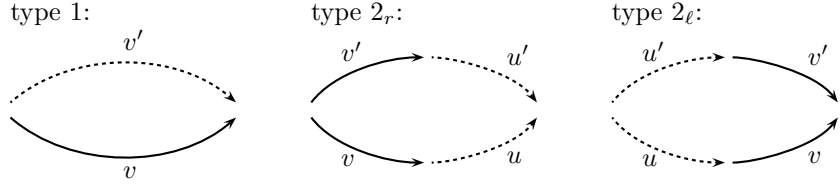

Special transformations are easily illustrated when words are viewed as paths in the Cayley diagram~\cite{Eps}, see Figure~\ref{F:Special}. With the above notation, the standard  result is that, for every positive presentation~$(\SS, \RR)$ and every word~$\ww$ of~$\WWW\SS$, we have
\begin{equation}
\label{E:Equiv}
\cl\ww = 1 \quad \Leftrightarrow \quad \ww \gozoiR \ew.
\end{equation}
Moreover, as is clear from Figure~\ref{F:Special}, every transformation of type~$2$ replaces a word with a new word that represents the same element of the involved group, so $\ww \gozotiR \ww'$ implies and, therefore, is equivalent to,~$\cl\ww  = \cl{\ww'}$. In particular, $\ww \gozotiR \ew$ is equivalent to~$\cl\ww = 1$. The question we will address is whether type~$\infty$ transformations can be avoided:

\begin{ques}
\label{Q:Final}
Assume that $(\SS, \RR)$ is a positive presentation and $\ww$ is a word of~$\WWW\SS$ that represents~$1$ in~$\GR\SS\RR$. Does $\ww \gozotR \ew$ necessarily hold?
\end{ques}

Introducing type~$2$ transformations really changes the situation. For instance, if $\SS$ is~$\{\tta, \ttb, \ttc\}$ and $\RR$ consists of $\tta \ttb = \ttb \tta$, $\ttb \ttc = \ttc \ttb$, and $\tta \ttc = \ttc \tta$, we observed above that, for $\ww = \tta \ttB \ttc \ttA \ttb \ttC$, although $\ww$ represents~$1$, the relation $\ww \gozoR \ew$ fails. Now, $\ww \gozotR \ew$ holds, since we can write
$\tta \ttB \underline{\ttc \ttA} \ttb \ttC
\goesR2 \tta \underline{\ttB \ttA} \ttc \ttb \ttC
\goesR1 \tta \ttA \ttB \ttc \ttb \ttC
\goes0 \ttB \underline{\ttc \ttb} \ttC
\goesR1 \ttB \ttb \ttc \ttC
\goes0 \ttc \ttC
\goes0 \ew$.

We are ready to introduce our main definition and state our conjecture.

\begin{defi}
We say that a positive presentation~$(\SS, \RR)$ satisfies Property~$\Prop$ if the relation $\ww \gozotR \ew$ holds for every word~$\ww$ of~$\WWW\SS$ that represents~$1$ in~$\GR\SS\RR$.
\end{defi}

\begin{conj}
\label{C:Main}
Every Artin--Tits presentation satisfies Property~$\Prop$.
\end{conj}

Before going on, let us immediately note that some positive presentations do \emph{not} satisfy Property~$\Prop$.
For instance, let $\SS$ be $\{\tta, \ttb, \ttc, \ttd, \tte, \ttf\}$ and let $\RR$ consist of the four relations $\ttc \tta = \tta \ttc \tte$, $\ttc \ttb = \ttb \ttc \tte$, $\ttd \tta = \tta \ttd \ttf$, and $\ttd \ttb = \ttb \ttd \ttf$. Let $\ww = \ttA \ttC \ttd \tta \ttB \ttD \ttc \ttb$. Then, as shown in Figure~\ref{F:Counter}, the word~$\ww$ represents~$1$ in the group~$\GR\SS\RR$, but $\ww \gozotR \ew$ cannot be true, as $\ww$ is eligible for no transformation of type~$0$, $1$, or~$2$. Hence the positive presentation~$(\SS, \RR)$ does not satisfy Property~$\Prop$.

\begin{figure}[htb]
\begin{picture}(60,24)(0,-1)
\pcline{<-}(1,10)(14,10)
\taput{$\tte$}
\pcline{<-}(16,10)(29,10)
\taput{$\ttc$}
\pcline{->}(31,10)(44,10)
\taput{$\ttd$}
\pcline{->}(46,10)(59,10)
\taput{$\ttf$}
\pcline{->}(30,19)(30,11)
\trput{$\tta$}
\pcline{<-}(30,9)(30,1)
\trput{$\ttb$}
\pcline{<-}(16,20)(29,20)
\taput{$\ttc$}
\psbezier{<-}(0,11)(0,20)(10,20)(14,20)
\put(4,20){$\tta$}
\pcline{<-}(16,0)(29,0)
\tbput{$\ttc$}
\psbezier{<-}(0,9)(0,0)(10,0)(14,0)
\put(4,-2){$\ttb$}
\pcline{->}(31,20)(44,20)
\taput{$\ttd$}
\psbezier{->}(46,20)(50,20)(60,20)(60,11)
\put(56,20){$\tta$}
\pcline{->}(31,0)(44,0)
\tbput{$\ttd$}
\psbezier{->}(46,0)(50,0)(60,0)(60,9)
\put(56,-2){$\ttb$}
\end{picture}
\caption[]{\sf\smaller Whenever $\ttc \tta = \tta \ttc \tte$, $\ttc \ttb = \ttb \ttc \tte$, $\ttd \tta = \tta \ttd \ttf$, and $\ttd \ttb = \ttb \ttd \ttf$ are valid relations, the word $\ww = \ttA \ttC \ttd \tta \ttB \ttD \ttc \ttb$ represents~$1$. However it is eligible for no special transformation, hence cannot satisfy $\ww \gozotR \ew$.}
\label{F:Counter}
\end{figure}
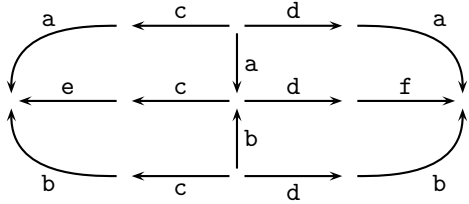

%%%%

\subsection{Connection with the Dehn algorithm and hyperbolicity}
\label{SS:Dehn}

Property~$\Prop$ is reminiscent of the property that underlies the Dehn algorithm, and, in this way, it is connected with hyperbolicity. Here, we make this connection precise. The result is stated in Proposition~\ref{P:Connection} below, and it shows that Property~$\Prop$ can be viewed as a weak hyperbolicity condition. In this way, Conjecture~\ref{C:Main} asserts that Artin--Tits groups, which cannot to be hyperbolic in general as they include a free Abelian group of rank two, nevertheless satisfy some weak variant of hyperbolicity.

Hereafter, we denote by~$\LG\ww$ the length (number of letters) of a word~$\ww$. Let us first recall the standard definition of the Dehn algorithm, here in the context of a positive presentation. 
 
\begin{defi} 
\label{D:Dehn}
Assume that $(\SS, \RR)$ a positive group presentation. For~$\ww, \ww'$ in~$\WWW\SS$, we say that $\ww'$ is obtained from~$\ww$ by a \emph{Dehn transformation with respect to~$\RR$} if $\ww'$ is obtained by replacing some factor~$\uu$ of~$\ww$ with a new factor~$\uu'$ such that $\LG{\uu} > \LG{\uu'}$ holds and there exists a relation~$\vv = \vv'$ of~$\RR$ such that $\uu\inv \uu'$ is a cyclic permutation of~$\vv\inv \vv'$ or of $\vv'{}\inv \vv$. We say that $\ww \gozDR \ww'$ holds if $\ww'$ is obtained from~$\ww$ by a finite sequence of type~$0$ transformations and Dehn transformations with respect to~$\RR$. 
\end{defi}

Clearly, $\ww \gozDR \ww'$ implies $\ww \eqR \ww'$ and, in particular, $\ww \gozDR \ew$ always implies $\ww \eqR \ew$. Saying that the latter implication is an equivalence means that the Dehn algorithm is valid for the presentation~$(\SS, \RR)$. It is well known \cite{GHV,Can,Lys} that, in this case, and provided the presentation~$(\SS, \RR)$ is finite, the group~$\GR\SS\RR$ is hyperbolic in the sense of Gromov.

It turns out that, at least in the case of positive presentations that satisfy simple length constraints, the relations~$\gozotR$ and~$\gozDR$ are closely connected.

\begin{prop}
\label{P:Connection}
Assume that $(\SS, \RR)$ is a positive presentation such that, for every relation~$\vv = \vv'$ of~$\RR$, we have $\big\vert\LG\vv - \LG{\vv'}\big\vert \le 2$. Then, for all~$\ww, \ww'$ in~$\WWW\SS$, the relation $\ww \gozDR \ww'$ implies~$\ww \gozotR \ww'$.
\end{prop}

\begin{proof} 
It is enough to prove that $\ww \gozotR \ww'$ holds when $\ww'$ is obtained from~$\ww$ by one Dehn transformation. As $\ww \gozotR \ww'$ implies $\ww_1 \ww \ww_2 \gozotR \ww_1 \ww' \ww_2$ for all~$\ww_1, \ww_2$, it is even sufficient to prove $\uu \gozotR \uu'$ when there exists a relation~$\vv = \vv'$ of~$\RR$ such that $\LG\uu > \LG{\uu'}$ holds and $\uu\inv \uu'$ is a cyclic permutation of~$\vv'{}\inv \vv$. The case when $\uu\inv \uu'$ is a cyclic permutation of~$\vv\inv \vv'$ would of course be symmetric. 

We consider the possible repartitions of the positive and negative letters in~$\uu$ and~$\uu'$. If $\uu$ contains at most one positive and one negative letter, it must contain a factor~$\ss\inv \tt$ or a factor~$\ss \tt\inv$, in which cases it is eligible for a type~$2$ special transformation, and we directly obtain $\uu \gotR \uu'$. Otherwise, assume that $\uu$ contains only positive letters, that is, it is a factor of~$\vv$. If $\uu$ is all of~$\vv$, then $\uu'$ is~$\vv'$, and we obtain $\uu \gooR \uu'$. Otherwise, we must have $\LG\uu < \LG\vv$, whence $\LG\uu \le \LG\vv - 1$, and $\LG{\uu'} \le \LG\vv - 2$. By hypothesis, we have $\LG\uu + \LG{\uu'} = \LG\vv + \LG{\vv'}$, and we deduce $\LG\vv + \LG{\vv'} \le 2\LG\vv -3$, whence $\LG{\vv'} \le \LG\vv-3$, which is incompatible with the hypothesis on~$(\SS, \RR)$. Finally, the case when $\uu$ contains only negative letters, hence is a factor of~$\vv'{}\inv$, is symmetric.
\end{proof}

Thus, at least for presentations satisfying the length constraint of Proposition~\ref{P:Connection}, the relation~$\gozot$ is a weak version of~$\gozD$ and, therefore, Property~$\Prop$ is a weak version of hyperbolicity. Artin--Tits presentations are eligible for Proposition~\ref{P:Connection} since all Artin--Tits relations are of the form $\vv = \vv'$ with $\LG\vv = \LG{\vv'}$. As, except in the degenerate case of free groups, Artin--Tits groups are not hyperbolic, Conjecture~\ref{C:Main} is a weaker statement in the same direction.

\begin{rema}
As it involves the equivalence relation~$\eqR$, Property~$\Prop$ is connected with the word problem of the considered presentation. However, Property~$\Prop$ need \emph{not} imply the decidability of the word problem. Indeed, starting from a word~$\ww$, there exist in general infinitely many words~$\ww'$ satisfying $\ww \gozotR \ww'$. For instance, in the case of $4$-strand braids, that is, of the Artin--Tits presentation of type~$A_3$---which satisfies Property~$\Prop$ as will be established in Proposition~\ref{P:Spherical} below---there exist infinitely many words~$\ww'$ satisfying $\sigma_2\inv \sigma_1\inv \sigma_3 \sigma_2 \gozotR \ww'$~\cite{Dhg}. In such a case, there is a priori no method for establishing the failure of $\ww \gozotR \ew$, and therefore, when Property~$\Prop$ is satisfied, for establishing the failure of $\ww \eqR \ew$.  
\end{rema}

To complete the discussion of the connection between Property~$\Prop$ and the Dehn algorithm, we now briefly mention the approach of Appel and Schupp~\cite{ApS} in  extra-large type, and explain why it does not seem to naturally lead to a proof of Property~$\Prop$.

An Artin-Tits presentation~$(\SS,\RR)$ is said to be of \emph{extra-large type} if all relations of~$\RR$ have the form~$\vv = \vv'$ with $\LG\vv = \LG{\vv'} \ge 4$. Appel and Schupp \cite{ApS} prove that, if $(\SS,\RR)$ is an Artin-Tits presentation of extra-large type and, if $\RRh$ is the set of all relations~$\uu = \uu'$ where~$\uu, \uu'$ are $\RR$-equivalent words of~$\WWW\SS$ involving only two letters of~$\SS$ and such that $\uu\inv \uu'$ is cyclically reduced, then a word~$\ww$ of~$\WWW\SS$ represents~$1$ in~$\GR\SS\RR$ if and only if $\ww \gozDRh \ew$ holds (here we do not assume that the words~$\uu, \uu'$ are positive, so the presentation~$(\SS, \RRh)$ is not positive in general). In other words, the Dehn algorithm with respect to the family~$\RRh$ solves the word problem, a result that does not contradict the non-hyperbolicity of the group~$\GR\SS\RR$ as $\RRh$ is always infinite. Owing to Proposition~\ref{P:Connection}, it is natural to wonder whether the implication
\begin{equation}
\label{E:ExtraLarge}
\ww \gozDRh \ww' \quad \Rightarrow \quad \ww \gozotR \ww'
\end{equation}
is always true. A direct proof of this implication would provide a proof of Conjecture~\ref{C:Main}, at least in the extra-large type: if $\ww$ represents~$1$, then $\ww \gozDRh \ew$ holds, and we would deduce $\ww \gozotR \ew$. However, \eqref{E:ExtraLarge} need not be true, even in very simple cases, as the following counter-example shows.

\begin{exam}
\label{X:Conterex}
Consider $\SS = \{\tta, \ttb\}$ and let $\RR$ consist of the unique relation $\tta \ttb \tta \ttb = \ttb \tta \ttb \tta$. Then $(\SS, \RR)$ is an Artin--Tits presentation of extra-large type, and the associated Coxeter group is the eight element dihedral group. Let $\ww = \ttA \ttb^2 \tta \ttb \tta \ttB$ and $\ww' = \tta \ttb \tta \ttb \ttA$. Then $\ww$ and~$\ww'$ are $\RR$-equivalent, we have $\LG\ww = 7 > \LG{\ww'} = 5$. So $\ww \gozDRh \ww'$ holds. We claim that $\ww \gozotR \ww'$ fails. The argument is as follows. We consider the finite fragment~$\Gamma$ of the Cayley graph of~$\GR\SS\RR$ represented in plain lines in Figure~\ref{F:Cayley}: the vertices are the nine left-divisors of the element~$\tta \ttb \tta \ttb^2$ in the monoid~$\MON\SS\RR$. Then the word~$\ww$ is, in the sense of~\cite{Dfo} or~\cite[Chapter~IV]{Dhr}, traced from~$\tta$ in~$\Gamma$, that is, there exists inside~$\Gamma$ a path starting from the vertex labeled~$\tta$ whose edges wear as labels the successive letters of~$\ww$. As explained in the cited references, the family of all words traced from a given vertex in the fragment of the Cayley graph made of all left-divisors of a given element is closed under special transformations of types~$0$, $1$, and~$2$. Now the picture shows that $\ww'$ is \emph{not} traced from~$\tta$ in~$\Gamma$ and, therefore, $\ww \gozotR \ww'$ cannot hold.
\end{exam}

\begin{figure}[htb]
$$\begin{picture}(60,25)(0,0)
\pcline{->}(1,0)(14,0)
\taput{$\tta$}
\pcline{->}(16,0)(29,0)
\taput{$\ttb$}
\pcline{->}(31,0)(44,0)
\taput{$\tta$}
\pcline{->}(1,8)(14,8)
\taput{$\tta$}
\pcline{->}(16,8)(29,8)
\taput{$\ttb$}
\pcline{->}(31,8)(44,8)
\taput{$\tta$}
\pcline{->}(1,16)(14,16)
\tbput{$\tta$}
\pcline{->}(16,16)(29,16)
\tbput{$\ttb$}
\pcline{->}(31,16)(44,16)
\tbput{$\tta$}
\pcline{->}(0,15)(0,9)
\trput{$\ttb$}
\pcline{->}(0,7)(0,1)
\trput{$\ttb$}
\pcline{->}(45,15)(45,9)
\tlput{$\ttb$}
\pcline{->}(45,7)(45,1)
\tlput{$\ttb$}
\pcline[style=exist]{->}(16,17)(29,24)
\tbput{$\tta$}
\pcline[style=exist]{->}(31,24)(44,24)
\tbput{$\ttb$}
\pcline[style=exist]{->}(46,24)(59,24)
\tbput{$\tta$}
\pcline[style=exist]{->}(60,23)(60,17)
\tlput{$\ttb$}
\pcline[style=exist]{->}(46,8)(59,16)
\taput{$\tta$}
\psline[linewidth=2pt,linecolor=gray]{c->}(13,18)(-2,18)(-2,-2)(47,-2)(47,7)
\psline[linewidth=2pt,linecolor=gray, style=etc]{c->}(16,19)(29,26)(62,26)(62,16)(49,8)
\pscircle(0,16){0.5}
\pscircle(45,0){0.5}
\put(-4,18){$1$}
\put(-6,8){$\ww$}
\put(64,20){$\ww'$}
\end{picture}$$
\caption{\sf\smaller Proving that $\ww \rightsquigarrow \ww'$ fails in Example~\ref{X:Conterex}: the word~$\ww$ is traced from the vertex~$\tta$ in the fragment of the Cayley graph written in plain lines, and so is every word obtained from~$\ww$ using special transformations, whereas $\ww'$ is not traced from the vertex~$\tta$ in that fragment.}
\label{F:Cayley}
\end{figure}
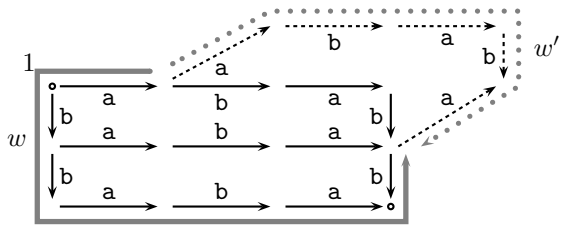

%%%%

\subsection{The case of spherical type}
\label{SS:Spherical}

Artin--Tits presentations of spherical type (also known as of finite type), that is, those such that the associated Coxeter group is finite~\cite{Hum}, provide natural test-cases: the corresponding groups are well understood and any conjecture involving general Artin--Tits groups can be expected to be provable for spherical type. This is the case for Property~$\Prop$.

\begin{prop}
\label{P:Spherical}
Every Artin--Tits presentation of spherical type satisfies Property~$\Prop$.
\end{prop}

In order to establish Proposition~\ref{P:Spherical} and to prepare the arguments of Subsection~\ref{SS:SphericalPP}, we recall some terminology about divisibility relations. If $\MM$ is a monoid and $\ff, \gg$ belong to~$\MM$, one says that $\ff$ is a \emph{left-divisor} of~$\gg$, and that $\gg$ is a \emph{right-multiple} of~$\ff$, if $\ff \gg' = \gg$ holds for some element~$\gg'$ of~$\MM$. If $\MM$ has no nontrivial (that is, distinct of~$1$) invertible element---as is the case of every monoid~$\MON\SS\RR$ defined by a positive presentation---the left-divisibility relation is a partial ordering on~$\MM$. Right-divisors are defined symmetrically: $\ff$ is a \emph{right-divisor} of~$\gg$, and $\gg$ is a \emph{left-multiple} of~$\ff$, if $\gg = \gg' \ff$ holds for some~$\gg'$. Finally, two elements of~$\MM$ are called \emph{left-coprime} (\resp.\ \emph{right-coprime}) if they admit no nontrivial common left-divisor (\resp.\ right-divisor). 

The specificity of the spherical type is the following classical result describing~$\GR\SS\RR$ as a group of (right)-fractions.

\begin{prop}\cite{BrS, Eps, Tat, Chb}
\label{P:Fraction}
Assume that $(\SS, \RR)$ is an Artin--Tits presentation of spherical type. Then the monoid~$\MON\SS\RR$ embeds in the group~$\GR\SS\RR$ and, for every element~$\gg$ of~$\GR\SS\RR$, there exists a unique pair $(\gg_1, \gg_2)$ of right-coprime elements of~$\MON\SS\RR$ satisfying $\gg = \gg_1 \gg_2\inv $. Moreover, if we have $\gg = \ff_1 \ff_2\inv$ with $\ff_1, \ff_2$ in~$\MON\SS\RR$, there exists~$\hh$ in~$\MON\SS\RR$ such that $\ff_\ie = \gg_\ie \hh$ holds for~$\ie = 1,2$.
\end{prop}

In the above context, the elements~$\gg_1$ and~$\gg_2$ above are called the \emph{right-numerator} and the \emph{right-denomin\-ator} of~$\gg$, here denoted by~$\NR\gg$ and~$\DR\gg$. In terms of words, there exists a simple connection with special transformations.

\begin{lemm}\cite{Dff}
\label{L:Rev}
Assume that $(\SS, \RR)$ is a Artin--Tits presentation of spherical type and $\ww$ is a word of~$\WWW\SS$. Let $\gg_1, \gg_2$ be the right numerator and denominator of~$\cl\ww$.

\ITEM1 There exist~$\ww_1, \ww_2$ in~$\WWWp\SS$ satisfying $\ww \goztR \ww_1 \ww_2\inv$ and $\cl{\ww_\ie} = \gg_\ie$ for $\ie = 1,2$.

\ITEM2 For all~$\ww'_1, \ww'_2$ in~$\WWWp\SS$ satisfying $\cl{\ww'_\ie} = \gg_\ie$ for $\ie = 1, 2$, we have $\ww \gozotR \ww'_1 \ww'_2{}\inv$.
\end{lemm}

\begin{proof}[Proof (sketch)]
\ITEM1 It follows from the fact that any two elements in the monoid~$\MON\SS\RR$ admit a common left-multiple that, starting from~$\ww$ and performing type~$0$ and~$2_\ell$ special transformations, one arrives in finitely many steps to a word of the form~$\vv_1\inv \vv_2$ with $\vv_1, \vv_2$ in~$\WWWp\SS$. By a symmetric argument, starting from~$\vv_1\inv \vv_2$ and performing type~$0$ and~$2_r$ special transformations, one arrives in finitely many steps to a word of the form~$\ww_1 \ww_2\inv$ with $\ww_1, \ww_2$ in~$\WWWp\SS$. So $\ww \goztR \ww_1 \ww_2\inv$ holds. Moreover, the element of~$\MON\SS\RR$ represented by~$\vv_1 \ww_1$ and~$\vv_2 \ww_2$ is the least common right-multiple of~$\cl{\vv_1}$ and~$\cl{\vv_2}$, which implies that $\cl{\ww_1}$ and $\cl{\ww_2}$ are right-coprime. The uniqueness result in Proposition~\ref{P:Fraction} then implies that $\ww_\ie$ represents~$\gg_\ie$ for $\ie = 1, 2$.

\ITEM2 Let $\ww'_\ie$ be any word of~$\WWWp\SS$ representing~$\gg_\ie$. The uniqueness result of Proposition~\ref{P:Fraction} implies that $\ww'_\ie \eqpR \ww_\ie$, where $\ww_1, \ww_2$ are the words obtained in~\ITEM1. Then we have $\ww_1 \ww_2\inv \gooR \ww'_1 \ww'_2{}\inv$. By composing, we deduce $\ww \gozotR \ww'_1 \ww'_2{}\inv$.
\end{proof}

Due to the symmetry of Artin--Tits relations, Proposition~\ref{P:Fraction} and Lemma~\ref{L:Rev} have left counterparts involving left-fractions of the form~$\gg_1\inv \gg_2$ and negative--positive words~$\ww_1\inv \ww_2$. We shall then speak of the left-numerator and the left-denominator of~$\gg$, denoted by~$\DL\gg$ and~$\NL\gg$. The latter are left-coprime elements of~$\MON\SS\RR$.

We can now easily complete the argument.

\begin{proof}[Proof of Proposition~\ref{P:Spherical}]
Assume that $\ww$ is a word of~$\WWW\SS$ that represents~$1$ in~$\GR\SS\RR$. Then the left numerator and denominator of~$\cl\ww$ are~$1$. By Lemma~\ref{L:Rev}, there exist~$\ww_1, \ww_2$ in~$\WWWp\SS$, both representing~$1$ in~$\MON\SS\RR$, and satisfying $\ww \goztR \ww_1 \ww_2\inv$. The only possibility for a word of~$\WWWp\SS$ to represent~$1$ is to be empty. So we must have $\ww \goztR \ew$ and, a fortiori, $\ww \gozotR \ew$. So Property~$\Prop$ is satisfied.
\end{proof}

\begin{rema}
\label{R:CounterArtin}
The result that $\ww \eqR \ew$ implies $\ww \goesR{0,2} \ew$ is false for general Artin--Tits presentations. For instance, consider the (right-angled) presentation~$(\SS, \RR)$ with $\SS = \{\tta, \ttb, \ttc, \ttd\}$ and $\RR = \{
\tta \ttc = \ttc \tta, \ttb \ttc = \ttc \ttb, \tta \ttd = \ttd \tta, \ttb \ttd = \ttd \ttb\}$.
The associated group is a direct product of two free groups of rank~$2$, respectively generated by~$\tta, \ttb$ and $\ttc, \ttd$. Let $\ww = \ttA \ttC \ttd \tta \ttB \ttD \ttc \ttb$. Then $\ww \gozotR \ew$ is true, as we can for instance write
\begin{multline*}
\hspace{10mm}\ttA \ttC \ttd \tta \ttB \ttD \ttc \ttb
\goesR1 \ttC \ttA \ttd \tta \ttB \ttD \ttc \ttb
\goesR1 \ttC \ttA \tta \ttd \ttB \ttD \ttc \ttb\\
\goes0 \ttC \ttd \ttB \ttD \ttc \ttb
\goesR1 \ttC \ttd \ttD \ttB \ttc \ttb
\goes0 \ttC \ttB \ttc \ttb
\goesR1 \ttC \ttB \ttb \ttc 
\goes0 \ttC \ttc
\goes0 \ew,\hspace{10mm}
\end{multline*}
hence $\ww$ represents~$1$ in~$\GR\SS\RR$. On the other hand, $\ww \goesR{0,2} \ew$ is false, as $\ww$ is eligible for no transformation of type~$0$ or~$2$. 
\end{rema}

%%%%

\subsection{An application}
\label{SS:Embed}

We conclude the section with an application of Property~$\Prop$ to the embeddability of the monoid~$\MON\SS\RR$ in the group~$\GR\SS\RR$. This application provides an additional motivation for investigating Property~$\Prop$, especially in the case of Artin--Tits presentations. As the proof appears in~\cite{Dia} (in a slightly different setting), we shall only state the result here. 

To this end, we first need to introduce a condition called completeness. We recall that $\WWWp\SS$ denotes the family of all words in~Ê$\WWW\SS$ that contain no letter~$\ss\inv$. Then, by definition, the monoid~$\MON\SS\RR$ (isomorphic to) $\WWWp\SS \quot \eqpR$, where $\eqpR$ is the least congruence on~$\WWWp\SS$ that includes~$\RR$. For $\ww, \ww'$ in~$\WWWp\SS$, by definition, $\ww \eqpR \ww'$ is equivalent to $\ww \gooR \ww'$ and, from there, to $\ww\inv \ww' \gozoR \ew$. Completeness is a formally similar condition that involves special transformations of type~$2\smallr$ instead of type~$1$. 

\begin{defi} \cite{Dgp, Dia}
\label{D:Completeness}
Assume that $(\SS, \RR)$ is a positive presentation. We say that $(\SS, \RR)$ is \emph{complete with respect to right-reversing} if, for all~$\ww, \ww'$ in~$\WWWp\SS$, the relation $\ww \eqpR \ww'$ is equivalent to $\ww\inv \ww' \goesR{0,2\smallr} \ew$.
\end{defi}

That $\ww\inv \ww' \goesR{0,2\smallr} \ew$ implies $\ww \eqpR \ww'$ is always true, so the non-trivial part is the converse implication, namely the condition that types~$0$ and~$2\smallr$ transformations are sufficient to detect all pairs of $\eqpR$-equivalent words. The embeddability result we wish to mention is as follows.

\begin{prop}\cite[Proposition~3.29]{Dia}
\label{P:Embed}
Assume that $(\SS, \RR)$ is a positive presentation that is complete with respect to right-reversing, contains no relation of the form $\ss \vv = \ss \vv'$ with $\vv \not= \vv'$, and satisfies Property~$\Prop$. Then the monoid~$\MON\SS\RR$ embeds in the group~$\GR\SS\RR$.
\end{prop}

The (non-trivial) proof consists in showing that, if two words~$\ww, \ww'$ of~$\WWWp\SS$ satisfy $\ww \eqR \ww'$, then one can go from~$\ww$ to~$\ww'$ using a certain ``fan'' of words of~$\WWWp\SS$ that are connected under~$\eqpR$. One then deduces that $\ww \eqR \ww'$ implies $\ww \eqpR \ww'$, which means that $\MON\SS\RR$ embeds in~$\GR\SS\RR$.

Note that, if $(\SS, \RR)$ is eligible for Proposition~\ref{P:Embed}, then the monoid~$\MON\SS\RR$ must be left- and right-cancellative, that is, $\ff \gg = \ff \gg'$ and $\gg \hh = \gg' \hh$ both imply $\gg = \gg'$. It is easy to see that completeness plus the hypothesis on relations $\ss \vv = \ss \vv'$ implies left-cancellativity. But these conditions alone say nothing about right-cancellativity. This shows that, in such a context, Property~$\Prop$ includes a hidden right-cancellativity condition.

For our current purpose, the point is that every Artin--Tits presentation is complete with respect to right-reversing \cite{Dgp}, a result  that in essence goes back to Garside~\cite{Gar} in the case of classical braids, that is, for Artin--Tits presentations of type~A, and to Brieskorn and Saito~\cite{BrS} in the general case. It follows that, if Artin--Tits presentations satisfy Property~$\Prop$, they are eligible for Proposition~\ref{P:Embed}. L.\,Paris~\cite{Par} proved that, if $(\SS, \RR)$ is an Artin--Tits presentation, then $\MON\SS\RR$ always embeds in~$\GR\SS\RR$. The beautiful but indirect argument of~\cite{Par} uses certain linear representations extending the Lawrence--Krammer representation of braids~\cite{Krb}, and a more direct proof remains to be discovered. Together with Proposition~\ref{P:Embed}, a proof of Conjecture~\ref{C:Main} would arguably provide such a natural proof.

%%%%%%%%

\section{Artin--Tits groups of type~FC}
\label{S:FC}

The family of Artin--Tits groups of type~FC is the smallest class of groups that contains all Artin--Tits groups of spherical type and is closed under forming an amalgamated product over a standard parabolic subgroup. The aim of this section is to prove

\begin{prop}
\label{P:FC}
Every Artin--Tits presentation of type~FC satisfies Property~$\Prop$.
\end{prop}

This result shows that Property~$\Prop$ is satisfied for a large number of Artin--Tits presentations. 

The section is organized as follows. In Subsection~\ref{SS:Propp}, we introduce a strengthening~$\Propp$ of Property~$\Prop$ that is more suitable for amalgamated products. Then we show that all Artin--Tits presentations of spherical type satisfy Property~$\Propp$ in Subsection~\ref{SS:SphericalPP}. Finally, we develop in Subsection~\ref{SS:Induction} the inductive argument needed to establish Proposition~\ref{P:FC}. 

The general scheme and the techniques used in this section follow those of Altobelli~\cite{Alt}. However, the argument has to be revisited completely because we are considering special transformations and the relation~$\gozotR$ instead of~$\eqR$. Also, in Subsection~\ref{SS:Induction}, we do not restrict to Artin--Tits groups and propose a new, hopefully optimized construction.

%%%%

\subsection{Property~$\Propp$}
\label{SS:Propp}

The definition of the family of Artin--Tits presentations of type~FC as the closure of the family of Artin--Tits presentations of spherical type under a certain type of amalgamated product naturally suggests an inductive proof for Proposition~\ref{P:FC}. However, we first need to modify Property~$\Prop$ to make the induction possible.

By definition of standard parabolic subgroups \cite{VdL}, the amalgamated products involved in the construction of type~FC groups correspond, in terms of presentations, to a union.

\begin{lemm}
\label{L:Amalg}
Say that two presentations $(\SS_1, \RR_1)$ and $(\SS_2, \RR_2)$ are in \emph{amalgam position} if, denoting by~$\GG'_\ie$ the subgroup of~$\GR{\SS_\ie}{\RR_\ie}$ generated by~$\SS_1 \cap \SS_2$, the identity map of~$\SS_1 \cap \SS_2$ induces an isomorphism of~$\GG'_1$ onto~$\GG'_2$. Then $(\SS_1 \cup \SS_2, \RR_1 \cup \RR_2)$ is a presentation of the amalgamated product $\GR{\SS_1}{\RR_1} *_{\GG'_1 = \GG'_2} \GR{\SS_2}{\RR_2}$.
\end{lemm}

It directly follows from the results of van der Lek~\cite{VdL} that two Artin--Tits presentations $(\SS_1, \RR_1)$, $(\SS_2, \RR_2)$ are in amalgam position if and only if the relations of~$\RR_1$ and of~$\RR_2$ that involve generators in~$\SS_1 \cap \SS_2$ coincide. Thus, owing to Proposition~\ref{P:Spherical}, which says that all Artin--Tits presentations of spherical type satisfy Property~$\Prop$, the natural approach for proving Proposition~\ref{P:FC} should be to establish that the union of two Artin--Tits presentations that are in amalgam position and satisfy Property~$\Prop$ again satisfies Property~$\Prop$. This however seems difficult, even probably false in general. What we shall do below is to introduce some strengthening~$\Propp$ of~$\Prop$ and apply the expected induction scheme to Property~$\Propp$ rather than to Property~$\Prop$.

Hereafter, if $(\SS, \RR)$ is a positive presentation and $\SS'$ is a subset of~$\SS$, we denote by~$\SG{\SS'}$ the subgroup of~$\GR\SS\RR$ generated by~$\SS'$. Similarly, we denote by~$\SGp{\SS'}$ the submonoid of~$\MON\SS\RR$ generated by~$\SS'$. If $\GG'$ is a subgroup of~$\GR\SS\RR$, a \emph{(left)-$\GG'$-transversal} is a subset of~$\GR\SS\RR$ that contains one element exactly in each left-$\GG'$-coset, and that contains~$1$---the latter condition is not always required, but here we explicitly include it in our definition. Below, we shall consider families of transversals with respect to various subgroups: we define an \emph{$\SS$-sequence of transversals} in~$\GR\SS\RR$ to be a sequence~$\TTT = (\TTT(\SS'))_{\SS' \ince \SS}$ such that $\TTT(\SS')$ is $\SG{\SS'}$-transversal in~$\GR\SS\RR$ for every~$\SS'$ included in~$\SS$.

\begin{defi}
\label{D:Propp}
Assume that $(\SS, \RR)$ is a positive presentation and $\TTT$ is an $\SS$-sequence of transversals in~$\GR\SS\RR$. We say that $(\SS, \RR)$ satisfies \emph{Property~$\Propp$ with respect to~$\TTT$} if, for all~$\SS', \SS_0$ included in~$\SS$ and~$\ww$ in~$\WWW\SS$, the following relation $\PPP_{\TTT}(\SS', \SS_0, \ww)$ is satisfied:
\begin{equation}
\cl\ww \in \SG{\SS'} \quad \mbox{implies} \quad \exists \vv \in \WWW{\SS'} \ \exists \uu \in \WWW{\SS' \cap \SS_0} \ (\ww \gozotR \vv \uu \mbox{\ and \ } \cl\vv \in \TTT(\SS_0)).
\end{equation}
\end{defi}

Roughly speaking, $\PPP_{\TTT}(\SS', \SS_0, \ww)$ says that, starting with a word~$\ww$ and using special transformations, we can isolate a representative of the (unique) element of the distinguished transversal~$\TTT(\SS_0)$ that lies in the left-$\SG{\SS_0}$-coset of~$\cl\ww$, and, if $\cl\ww$ lies in~$\SG{\SS'}$, we can moreover require that the final words lie in~$\WWW{\SS'}$. 

\begin{lemm}
\label{L:Implies}
Property~$\Propp$ implies Property~$\Prop$.
\end{lemm}

\begin{proof}
Assume that $(\SS, \RR)$ satisfies Property~$\Propp$ with respect to~$\TTT$, and that $\ww$ represents~$1$ in~$\GR\SS\RR$. Then $\cl\ww$ belongs to~$\SG{\emptyset}$. Hence $\PPP_{\TTT}(\emptyset, \SS, \ww)$ says that there exists~$\vv$ and~$\uu$ in~$\WWW{\emptyset}$ such that $\ww \gozotR \vv \uu$ holds. Now the only element of~$\WWW{\emptyset}$ is the empty word~$\ew$, so the only possibility is $\vv = \uu = \ew$, whence $\ww \gozotR \ew$. So $(\SS, \RR)$ satisfies Property~$\Prop$.
\end{proof}

The result we shall establish is then

\begin{prop}
\label{P:FCMain}
Every Artin--Tits presentation of type~FC satisfies Property~$\Propp$.
\end{prop}

By Lemma~\ref{L:Implies}, Proposition~\ref{P:FCMain} directly implies Proposition~\ref{P:FC}. The proof of Proposition~\ref{P:FCMain} will occupy the rest of this section, and consists in proving that every Artin--Tits presentation of spherical type satisfies Property~$\Propp$ (Proposition~\ref{P:SphericalP}) and that the union of two presentations that are in amalgam position and satisfy Property~$\Propp$ satisfies Property~$\Propp$ (Proposition~\ref{P:Union}). We begin with two direct consequences of the definition:

\begin{lemm}
\label{L:Subgroup}
Assume that $(\SS, \RR)$ is a positive presentation that satisfies Property~$\Propp$. 

\ITEM1 For all $\SS',\SS''$ included in~$\SS$, we have $\SG{\SS'}\cap \SG{\SS''} = \SG{\SS'\cap \SS''}$.

\ITEM2 For all~$\SS'$ included in~$\SS$ and~$\ww$ in~$\WWW\SS$, 
\begin{equation}
\label{E:Subgroup}
\cl\ww \in \SG{\SS'} \quad \mbox{implies} \quad \exists \vv \in \WWW{\SS'}\ (\ww \gozotR \vv).
\end{equation}
\end{lemm}

\begin{proof}
Assume that $(\SS, \RR)$ satisfies Property~$\Propp$ with respect to~$\TTT$. 

\ITEM1 The inclusion $\SG{\SS'\cap \SS''} \ince \SG{\SS'}\cap \SG{\SS''}$ is always true. For the other inclusion, assume that $\gg$ belongs to~$\SG{\SS'}\cap \SG{\SS''}$. Let $\ww$ be a word of~$\WWW\SS$ representing~$\gg$. By~$\PPP_{\TTT}(\SS', \SS'', \ww)$, there exist~$\vv$ in~$\WWW{\SS'}$ and~$\uu$ in~$\WWW{\SS' \cap \SS''}$ satisfying $\ww \gozotR \vv \uu$ and $\cl\vv \in \TTT(\SS'')$. Then we have $\gg = \cl\ww = \cl\vv \, \cl\uu$. By hypothesis $\cl\vv$ belongs to~$\SG{\SS''}$, so $\cl\vv$ is the element of the $\SG{\SS''}$-transversal~$\TTT(\SS'')$ that belongs to the left-$\SG{\SS''}$-coset of~$\cl\ww$. Now, by hypothesis, $\cl\ww$ belongs to~$\SG{\SS''}$, hence $\cl\vv$ must be~$1$, and, therefore, we have $\cl\ww = \cl\uu$, whence $\gg = \cl\uu \in \SG{\SS' \cap \SS''}$.

\ITEM2 Assume $\cl\ww \in \SG{\SS'}$. Then, by $\PPP_{\TTT}(\SS', \emptyset, \ww)$, there exist~$\vv$ in~$\WWW{\SS'}$ and~$\uu$ in~$\WWW\emptyset$ satisfying $\ww \gozotR \vv \uu$ and $\cl\vv \in \TTT(\emptyset)$. The only possibility is $\uu = \ew$, and, therefore, we have $\ww \gozotR \vv$. So \eqref{E:Subgroup} is satisfied.
\end{proof}

\subsection{The spherical case}
\label{SS:SphericalPP}

According to the scheme described above, the first step toward Proposition~\ref{P:FCMain} is to consider Artin--Tits presentations of spherical type. We shall establish

\begin{prop}
\label{P:SphericalP}
Every Artin--Tits presentation of spherical type satisfies Property~$\Propp$.
\end{prop}

A first observation is that, for Artin--Tits presentations, Condition~$\PPP_{\TTT}(\SS', \SS_0, \ww)$, which simultaneously involves two subsets~$\SS_0, \SS'$ of~$\SS$, can be split into two simpler conditions involving only one subset of~$\SS$ at a time.

\begin{lemm}
\label{L:Propp}
Assume that $(\SS, \RR)$ is an Artin--Tits presentation. Then $(\SS, \RR)$ satisfies Property~$\Propp$ with respect to~$\TTT$ if and only if, for all~$\SS_0$ included in~$\SS$ and $\ww$ in~$\WWW\SS$, we have~\eqref{E:Subgroup} and
\begin{equation}
\label{E:Proppp}
\exists \vv \in \WWW\SS \ \exists \uu \in \WWW{\SS_0} \ (\ww \gozotR \vv \uu \mbox{\ and \ } \cl\vv \in \TTT(\SS_0)).
\end{equation}
\end{lemm}

\begin{proof}
We already observed in Lemma~\ref{L:Subgroup} that Property~$\Propp$ implies~\eqref{E:Subgroup} for all~$\SS'$ and~$\ww$. On the other hand, \eqref{E:Proppp} is (equivalent to) $\PPP_{\TTT}(\SS, \SS_0, \ww)$, so Property~$\Propp$ implies~\eqref{E:Subgroup} and~\eqref{E:Proppp}.

Conversely, assume that $(\SS, \RR)$ is an Artin--Tits presentation that satisfies~\eqref{E:Subgroup} and~\eqref{E:Proppp} for all~$\SS', \SS_0$, and~$\ww$. Let~$\SS_0, \SS'$ be subsets of~$\SS$ and let $\ww$ be a word of~$\WWW\SS$ that represents an element of~$\SG{\SS'}$. By~\eqref{E:Subgroup}, there exists a word~$\ww'$ of~$\WWW{\SS'}$ satisfying $\ww \gozotR \ww'$ holds. Next, by~\eqref{E:Proppp}, we find $\vv$ in~$\WWW\SS$ and $\uu$ in~$\WWW{\SS_0}$ satisfying $\ww' \gozotR \vv\uu$ and $\cl\vv \in \TTT(\SS_0)$. Now, due to the specific syntactic form of Artin--Tits relations, the hypothesis that $\ww'$ lies in~$\WWW{\SS'}$ implies that $\ww''$ lies in~$\WWW{\SS'}$ for every word~$\ww''$ satisfying $\ww' \gozotR \ww''$. Indeed, a special transformation of type~$0, 1$, or~$2$ can never let a new letter of~$\SS$ or~$\SS\inv$ appear. Hence $\vv$ must lie in~$\WWW{\SS'}$, and $\uu$ must lie in~$\WWW{\SS_0 \cap \SS'}$. So $\PPP_{\TTT}(\SS', \SS_0, \ww)$ is satisfied.
\end{proof}

Next, we observe that \eqref{E:Subgroup} is almost trivial.

\begin{lemm}
\label{L:ProppS}
Assume that $(\SS, \RR)$ is an Artin--Tits presentation of spherical type. Then \eqref{E:Subgroup} is satisfied for all~$\SS_0$ included in~$\SS$ and $\ww$ in~$\WWW\SS$.
\end{lemm}

\begin{proof}
Assume that $\SS_0$ is included in~$\SS$ and $\ww$ is an element of~$\WWW\SS$ that represents an element of~$\SG{\SS_0}$. Then $\cl\ww$ has at least one representative that lies in~$\WWW{\SS_0}$ and, therefore, $\NR{\cl\ww}$ and~$\DR{\cl\ww}$ admit representatives, say~$\vv_1, \vv_2$ that lie in~$\WWWp{\SS_0}$. Then the word~$\vv_1 \vv_2\inv$ is a word of~$\WWW{\SS_0}$ and, by Lemma~\ref{L:Rev}\ITEM2, we have $\ww \gozotR \vv_1 \vv_2\inv$.
\end{proof}

So we are left with proving that, if $(\SS, \RR)$ is an Artin--Tits presentation of spherical type, then all instances of~\eqref{E:Proppp} are satisfied. We begin with a general auxiliary result that will be used several times in the sequel.

\begin{lemm}
\label{L:Lcm}
Assume that $(\SS, \RR)$ is an Artin--Tits presentation, $\SS_0$ is included in~$\SS$, and we have $\gg \hh = \gg' \hh'$ with $\gg, \gg'$ in~$\MON\SS\RR$ and $\hh, \hh'$ in~$\SGp{\SS_0}$. If $\gg$ is right-divisible by no nontrivial element of~$\SGp{\SS_0}$, then necessarily $\hh$ right-divides~$\hh'$.
\end{lemm}

\begin{proof}
(See Figure~\ref{F:Lcm}.) In the monoid~$\MON\SS\RR$, the elements~$\hh$ and~$\hh'$ admit a least common left-multiple, say~$\hh''$, that lies in~$\SGp{\SS_0}$. Since $\gg \hh$, which is also~$\gg' \hh'$, is a left-multiple of~$\hh$ and~$\hh'$, it is a left-multiple of~$\hh''$, say $\gg \hh = \gg'' \hh''$ with $\gg'' \in \MON\SS\RR$. Write $\hh'' = \ff \hh$. We find $\gg \hh = \gg'' \ff \hh$, whence $\gg = \gg'' \ff$. By construction, $\ff$ belongs to~$\SGp{\SS_0}$ and right-divides~$\gg$, so the hypothesis that $\gg$ is right-divisible by no nontrivial element of~$\SGp{\SS_0}$ implies $\ff = 1$, whence $\hh'' = \hh$. The latter equality means that $\hh'$ right-divides~$\hh$. 
\end{proof}

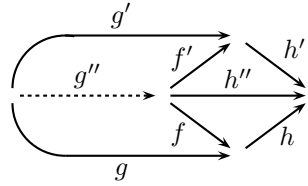
\begin{figure}[htb]
$$\begin{picture}(40,16)(0,1)
\pcline[style=exist]{->}(1,8)(19,8)
\taput{$\gg''$}
\pcline{->}(21,8)(39,8)
\taput{$\hh''$}
\psarc(7,7){7}{180}{270}
\pcline{->}(7,0)(29,0)
\tbput{$\gg$\qquad}
\psarc(7,9){7}{80}{180}
\pcline{->}(7,16)(29,16)
\taput{$\gg'$\qquad}
\pcline{->}(31,15)(39,9)
\put(36,13){$\hh'$}
\pcline{->}(31,1)(39,7)
\put(35.5,1.5){$\hh$}
\pcline{->}(21,9)(29,15)
\put(21,12){$\ff'$}
\pcline{->}(21,7)(29,1)
\put(21,2){$\ff$}
\end{picture}$$
\caption{\sf\smaller Proof of Lemma~\ref{L:Lcm}.}
\label{F:Lcm}
\end{figure}

Until the end of this subsection, we assume that $(\SS, \RR)$ is an Artin--Tits presentation of spherical type. We denote by~$\GG$ the group~$\GR\SS\RR$, and by~$\GGp$ the monoid~$\MON\SS\RR$. We shall now construct an $\SS$-sequence of transversals~$\TTT$ with the aim of proving that \eqref{E:Proppp} is always satisfied.

\begin{defi}
Assume that $\SS_0$ is included in~$\SS$. An element~$\gg$ of~$\GG$ is called \emph{$\SS_0$-minimal} if, for every~$\hh$ in~$\SG{\SS_0}$, we have 
$\LG{\DL{\gg\hh}} \ge \LG{\DL\gg}$ and $\LG{\DL{\gg\hh}} = \LG{\DL\gg}$ implies $\LG{\NL{\gg\hh}} \ge\nobreak \LG{\NL\gg}$. Then we put
$$\TTT(\SS_0) = \{\gg \in \GG \mid \mbox{$\gg$ is $\SS_0$-minimal}\}.$$
\end{defi}

We shall show that $\TTT$ is an $\SS$-sequence of transversals witnessing for Property~$\Propp$. We begin with several technical results about $\SS_0$-minimal elements.

\begin{lemm}
\label{L:Minimal}
\ITEM1 An element of~$\GGp$ is $\SS_0$-minimal if and only it is right-divisible by no nontrivial element of~$\SGp{\SS_0}$.

\ITEM2 If $\gg$ is $\SS_0$-minimal, then, for every~$\hh$ in~$\SGp{\SS_0}$, the elements $\DL\gg$ and~$\NL\gg \hh$ are left-coprime.

\ITEM3 If $\gg$ is $\SS_0$-minimal and lies in~$\GGp$, then, for every~$\hh$ in~$\SG{\SS_0}$, the elements $\NR{\gg \hh}$ and~$\DR{\hh}$ are right-coprime.
 
\ITEM4 If $\gg$ is $\SS_0$-minimal, then so is $\NL\gg$.
\end{lemm} 

\begin{proof}
\ITEM1 Assume that $\gg$ is an element of~$\GGp$ and $\hh$ is an element of~$\SGp{\SS_0}$ that right-divides~$\gg$, say $\gg = \gg' \hh$ with~$\gg'$ in~$\GGp$. Then we have $\LG{\DL{\gg \hh\inv}} = \LG{\DL{\gg}} = 0$ and $\LG{\NL{\gg \hh\inv}} = \LG{\gg'} = \LG\gg - \LG{\hh}$. If $\gg$ is $\SS_0$-minimal, the only possibility is $\LG{\hh} = 0$, that is, $\hh = 1$. 

Conversely, assume that $\gg$ is an element of~$\GGp$ that is right-divisible by no nontrivial element of~$\SGp{\SS_0}$. Let~$\hh$ belong to~$\SG{\SS_0}$. We have $\DL\gg = 1$, so $\LG{\DL{\gg \hh}} \ge \LG{\DL\gg}$ is always satisfied. Assume $\LG{\DL{\gg \hh}} = \LG{\DL\gg}$, hence $\LG{\DL{\gg \hh}} = 0$, that is, $\gg \hh \in \GGp$. Let $\gg' = \gg \hh$, and write $\hh = \hh_1 \hh_2\inv$ with $\hh_1, \hh_2$ in~$\SGp{\SS_0}$. Then we have $\gg \hh_1 = \gg' \hh_2$. By Lemma~\ref{L:Lcm}, $\hh_2$ right-divides~$\hh_1$, that is, $\hh$ belongs to~$\SGp{\SS_0}$. Then we obtain $\LG{\NL{\gg \hh}} = \LG{\gg \hh} \ge \LG\gg = \LG{\NL\gg}$, and $\gg$ is $\SS_0$-minimal.

\ITEM2 Assume that $\gg$ is $\SS_0$-minimal and $\hh$ lies in~$\SGp{\SS_0}$. Put $\gg_1 = \DL\gg$, $\gg_2 = \NL\gg$. By definition of an $\SS_0$-minimal element, we have $\LG{\DL{\gg \hh}} \ge \LG{\DL\gg} = \LG{\gg_1}$. Assume that $\ff$ is a common left-divisor of~$\gg_1$ and~$\gg_2 \hh$, say $\gg_1 = \ff \gg'_1$ and $\gg_2 \hh = \ff \gg'_2$ with $\gg'_1, \gg'_2$ in~$\GGp$. As we have $\gg \hh = \gg'_1{}\inv \gg'_2$, we deduce $\LG{\DL{\gg \hh}} \le \LG{\gg'_1} = \LG{\gg_1} - \LG\ff$. The only possibility is $\LG\ff = 0$, that is, $\ff = 1$. Hence $\gg_1$ and~$\gg_2 \hh$ are left-coprime.

\ITEM3 Let $\hh_1 = \NR{\hh}$ and $\hh_2 = \DR{\hh}$. Assume that $\ff$ right-divides~$\gg \hh_1$ and~$\hh_2$, say $\gg \hh_1 = \gg' \ff$. As $\ff$ right-divides~$\hh_2$, it lies in~$\SGp{\SS_0}$. As $\gg$ is $\SS_0$-minimal, \ITEM1 implies that it is right-divisible by no nontrivial element of~$\SGp{\SS_0}$. Then Lemma~\ref{L:Lcm} implies that $\ff$ right-divides~$\hh_1$ and, therefore, the hypothesis that $\hh_1$ and~$\hh_2$ are right-coprime implies $\ff = 1$. So $\gg \hh_1$ and $\hh_2$ are right-coprime, as expected.

\ITEM4 Assume that $\gg$ is $\SS_0$-minimal. Let $\gg_1 = \DL\gg$ and $\gg_2 = \NL\gg$. Assume $\gg_2 = \gg' \hh$ with $\hh$ in~$\SGp{\SS_0}$. Let $\gg' = \gg_1\inv \gg' = \gg \hh\inv$. The hypothesis that $\gg_1$ and~$\gg_2$ are left-coprime implies that $\gg_1$ and~$\gg'$ are left-coprime. Hence we have $\DL{\gg'} = \gg_1 = \DL\gg$ and $\NL{\gg'} = \gg'$. As we have $\gg' = \gg \hh\inv$ with $\hh\inv \in \SG{\SS_0}$ and $\LG{\DL{\gg'}} = \LG{\DL\gg}$, the hypothesis that $\gg$ is $\SS_0$ minimal implies $\LG{\NL{\gg'}} \ge \LG{\NL\gg}$, that is, $\LG{\gg'} \ge \LG{\gg_2}$. As we have $\LG{\gg_2} = \LG{\gg'} + \LG{\hh}$, the only possibility is $\LG{\hh} = 1$, that is, $\hh = 1$. Thus $\gg$ has no nontrivial right-divisor in~$\SG{\SS_0}$. By~\ITEM1, $\gg$ is $\SS_0$-minimal.
\end{proof}

\begin{lemm}
For every~$\SS_0$ included in~$\SS$, the family~$\TTT(\SS_0)$ is a $\SG{\SS_0}$-transversal in~$\GG$.
\end{lemm}

\begin{proof}
Let $\CC$ be a left-$\SG{\SS_0}$-coset. Let~$\CC'$ be the set
$\{\gg \in \CC \mid \mbox{$\LG{\DL\gg}$ is minimal}\}$, and $\CC''$ be $\{\gg \in \CC' \mid \mbox{$\LG{\NL\gg}$ is minimal}\}$. By definition, $\CC''$ is nonempty, and every element of~$\CC''$ is an element of~$\CC$ that is $\SG{\SS_0}$-minimal. So every left-$\SG{\SS_0}$-coset contains at least one $\SS_0$-minimal element.

Assume now that $\gg$ and $\gg'$ belong to~$\TTT(\SS_0)$ and lie in the same left-$\SG{\SS_0}$-coset, say $\gg = \gg' \hh$ with $\hh \in \SG{\SS_0}$. Let $\gg_1, \gg_2$ be the left denominator and numerator of~$\gg$, and $\gg'_1, \gg'_2$ be those of~$\gg'$. Write $\hh = \hh_1 \hh_2\inv$ with $\hh_1, \hh_2$ in~$\SGp{\SS_0}$. As $\gg'$ is $\SS_0$-minimal, Lemma~\ref{L:Minimal}\ITEM2 implies that $\gg'_1$ and~$\gg'_2 \hh_2$ are left-coprime. Now we have $\gg_1\inv (\gg_2 \hh_1) = \gg'_1{}\inv (\gg'_2 \hh_2)$. As $\gg'_1$ and~$\gg'_2 \hh_2$ are left-coprime, $\gg'_1$ must right-divide~$\gg_1$. By a symmetric argument exchanging the roles of~$\gg$ and~$\gg'$ and those of~$\hh_1$ and~$\hh_2$, we see that $\gg_1$ must right-divide~$\gg'_1$. We deduce $\gg_1 = \gg'_1$, whence $\gg_2 \hh_1 = \gg'_2 \hh_2$. By Lemma~\ref{L:Minimal}\ITEM4, $\gg_2$ is $\SS_0$-minimal, so, by Lemma~\ref{L:Minimal}\ITEM1, $\gg_2$ is right-divisible by no nontrivial element of~$\SGp{\SS_0}$. Hence, by Lemma~\ref{L:Lcm}, $\hh_2$ right-divides~$\hh_1$. A symmetric argument shows that $\hh_1$ right-divides~$\hh_2$. We deduce $\hh_1 = \hh_2$, whence $\gg = \gg'$. So every left-$\SG{\SS_0}$-coset contains at most one $\SS_0$-minimal element.

Finally, it is clear that $1$ is $\SS_0$-minimal, hence $\TTT(\SS_0)$ contains~$1$.
\end{proof}

With these results at hand, we can complete the argument.

\begin{proof}[Proof of Proposition~\ref{P:SphericalP}]
We have seen that the only point that remains to be proved is the satisfaction of~\eqref{E:Proppp} for all~$\SS_0$ and~$\ww$. So, assume that $\SS_0$ are included in~$\SS$ and $\ww$ is a word of~$\WWW\SS$. Put $\gg = \cl\ww$. As $\TTT(\SS_0)$ is $\SG{\SS_0}$-transversal, we have $\gg = \gg' \hh$ for some (unique)~$\gg'$ in~$\TTT(\SS_0)$ and~$\hh$ in~$\SG{\SS_0}$. In~$\GG$, we have
\begin{equation}
\label{E:Spherical}
\DL\gg\inv \NL\gg \DR{\hh}
= \gg \DR{\hh} 
= \gg' \hh \DR{\hh}
= \DL{\gg'}\inv \NL{\gg'} \NR{\hh}.
\end{equation}
By hypothesis, $\gg'$ is $\SS_0$-minimal. Hence, by Lemma~\ref{L:Minimal}\ITEM2, $\DL{\gg'}$ and $\NL{\gg'} \NR{\hh}$ are left-coprime. So, by (the left counterpart of) Proposition~\ref{P:Fraction}, \eqref{E:Spherical} implies the existence of~$\ff$ in~$\GGp$ satisfying 
$$\DL\gg = \ff \DL{\gg'} \mbox{\quad and \quad} \NL\gg \DR{\hh} = \ff \NL{\gg'} \NR{\hh},$$
see Figure~\ref{F:SphericalPlus}. By Lemma~\ref{L:Minimal}\ITEM4, the hypothesis that $\gg'$ is $\SS_0$-minimal implies that $\NL{\gg'}$ is $\SS_0$-minimal and, then, 
Lemma~\ref{L:Minimal}\ITEM3 implies that $\NL{\gg'} \NR{\hh}$ and $\DR{\hh}$ are right-coprime.

Now, let $\ww_1, \ww_2$ be words of~$\WWWp\SS$ representing the left denominator and numerator of~$\gg$, let $\vv_1, \vv_2$ be words of~$\WWWp\SS$ representing the left denominator and numerator of~$\gg'$, and let~$\uu_1, \uu_2$ be words of~$\WWWp{\SS_0}$ representing the right numerator and denominator of~$\hh$. Finally, let $\uu$ represent~$\ff$. First, by (the left counterpart of) Lemma~\ref{L:Rev}, we have $\ww \gozotR \ww_1\inv \ww_2$. Next, $\NL\gg = \ff \NL{\gg'}$ implies $\ww_1 \eqpR \uu \vv_1$, whence $\ww_1\inv \gooR \vv_1\inv \uu\inv$ and $\ww \gozotR \vv_1\inv \uu\inv \ww_2$. Finally, we have $\ff\inv \NL\gg = \NL{\gg'} \NR{\hh} \DR{\hh}\inv$, and we saw above that $\NL{\gg'} \NR{\hh}$ and $\DR{\hh}$ are right-coprime, which implies that $\NL{\gg'} \NR{\hh}$ is the right-numerator of $\ff\inv \NL\gg$ and $\DR{\hh}$ is its right-denominator. By definition, $\uu\inv \ww_2$ represents~$\ff\inv \NL\gg$, whereas $\vv_2 \uu_1$ represents~$\NL{\gg'} \NR{\hh}$ and $\uu_2$ represents~$\DR{\hh}$. Lemma~\ref{L:Rev}\ITEM2 then implies $\uu\inv \ww_2 \gozotR \vv_2 \uu_1 \uu_2\inv$, whence $\ww \gozotR \vv_1\inv \vv_2 \uu_1 \uu_2\inv$, that is, $\ww \gozotR \vv \uu$ if we put $\vv = \vv_1\inv \vv_2$ and $\uu = \uu_1 \uu_2\inv$. Now, by construction, $\vv$ represents~$\gg'$, an element of~$\TTT(\SS_0)$, whereas $\uu$ is a word of~$\WWW{\SS_0}$. So \eqref{E:Proppp} is satisfied.
\end{proof}

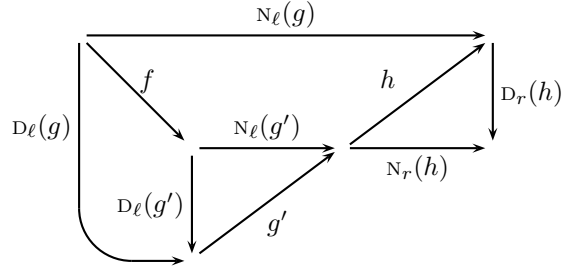
\begin{figure}[htb]
$$\begin{picture}(55,35)(0,0)
\pcline{->}(15,14)(15,1)
\tlput{$\DL{\gg'}$}
\pcline{->}(16,15)(34,15)
\taput{$\NL{\gg'}$}
\pcline{->}(36,15)(54,15)
\tbput{$\NR{\hh}$}
\pcline{->}(55,29)(55,16)
\trput{$\DR{\hh}$}
\pcline{->}(1,30)(54,30)
\taput{$\NL\gg$}
\pcline{->}(7,0)(14,0)
\psarc(7,7){7}{180}{270}
\pcline(0,29)(0,7)
\tlput{$\DL\gg$}
\pcline{->}(1,29)(14,16)
\put(8,23){$\ff$}
\pcline{->}(16,1)(34,14.5)
\put(25,4){$\gg'$}
\pcline{->}(36,15.5)(54,29)
\put(40,23){$\hh$}
\end{picture}$$
\caption{\sf\smaller Proof of Proposition~\ref{P:SphericalP}: the point is that $\DL{\gg'}$ and $\NL{\gg'} \NR{\hh}$ are left-coprime, and that $\NL{\gg'} \NR{\hh}$ and $\DR{\hh}$ are right-coprime.}
\label{F:SphericalPlus}
\end{figure}

%%%%

\subsection{Preservation under union}
\label{SS:Induction}

We shall now prove that Property~$\Propp$ is preserved under suitable unions. 

\begin{prop}
\label{P:Union} 
Assume that $(\SS_1, \RR_1)$ and $(\SS_2, \RR_2)$ are positive group presentations that are in amalgam position and satisfy Property~$\Propp$. Then $(\SS_1\cup\SS_2,\RR_1\cup \RR_2)$ satisfies Property~$\Propp$. 
\end{prop} 

As explained in Subsection~\ref{SS:Propp}, this result together with Proposition~\ref{P:SphericalP} implies that every Artin--Tits presentation of FC type satisfies Property~$\Propp$, hence a fortiori Property~$\Prop$, thus completing the proof of Proposition~\ref{P:FC}.

We split the proof of Proposition~\ref{P:Union} into several steps. Until the end of the section, we assume that $(\SS_1,\RR_1)$ and $(\SS_2,\RR_2)$ are fixed positive presentations that are in amalgam position and satisfy Property~$\Propp$. For~$\ie = 1,2$, we denote by~$\GG_\ie$ the group~$\GR{\SS_\ie}{\RR_\ie}$, we write~$\GG_{1,2}$ for~$\GG_1 \cap \GG_2$, and~$\GG$ for the amalgamated product $\GG_1 *_{\GG_{1,2}} \GG_2$. Thus, if we put $\SS = \SS_1\cup \SS_2$, $\RR = \RR_1\cup \RR_2$, and $\SS_{1,2} = \SS_1\cap \SS_2$, then, by hypothesis, we have $\GG = \GR\SS\RR$ and $\GG_{1,2} = \SG{\SS_{1,2}}$ (both in the sense of~$\GG_1$ and of~$\GG_2$). Next, we fix an $\SS_\ie$-sequence of transversals~$\TTT_\ie$ witnessing that $(\SS_\ie, \RR_\ie)$ satisfies Property~$\Propp$, and write $\PPPPe{\SS'}{\SS''}\ww$ for $\PPP_{\TTT_\ie}(\SS', \SS'', \ww)$. 

For~$\gg$ in~$\GG_1 \bigtriangleup \GG_2$ (that is, in $\GG_1 \setminus \GG_2$ or in~$\GG_2 \setminus \GG_1$), we denote by~$\oo\gg$ the index~$\ie$ such that $\gg$ lies in~$\GG_\ie$. It follows from Lemma~\ref{L:Amalg} that, for every element~$\gg$ of~$\GG$, there exists a unique sequence $(\gg_1 \wdots \gg_\rr, \ggs)$ such that $\gg = \gg_1 \pdots \gg_\rr \ggs$ holds, for every~$\ii$ the element~$\gg_\ii$ is a nontrivial element of~$\TTT_1(\SS_{1,2}) \cup \TTT_2(\SS_{1,2})$ with $\oo{\gg_\ii} \not= \oo{\gg_{\ii+1}}$, and $\ggs$ is an element of~$\GG_{1,2}$. 

\begin{defi}
\label{D:CanDec}
In the above context, the sequence~$(\gg_1 \wdots \gg_\rr, \ggs)$ is called the \emph{amalgam decomposition} of~$\gg$ and denoted by~$\AD\gg$. The number~$\rr$ is called the \emph{amalgam length} of $\gg$ and denoted by~$\AL\gg$. Finally, the \emph{tail}~$\tail\gg$ of~$\gg$ is defined to be~$\gg_\rr \ggs$ for $\rr \ge 1$, and to be~$\ggs$ (that is, $\gg$) for $\rr = 0$. 
\end{defi}

(Note that $\tail\gg$ is always $\gg_\rr \ggs$ if $\AD\gg$ is extended with~$\gg_\ii = 1$ for $\ii \le 0$.) We begin with a natural result saying that, if a word~$\ww$ represents an element~$\gg$ of~$\GG$, then, by using special transformations, we can go from~$\ww$ to a concatenation of words that represent the successive entries of the amalgam decomposition of~$\gg$. So, roughly speaking, the amalgam decomposition can be computed using special transformations. 

\begin{lemm}
\label{L:Pieces} 
Assume that $\ww$ is a word of~$\WWW\SS$ and $(\gg_1 \wdots \gg_\rr, \ggs)$ is the amalgam decomposition of~$\cl\ww$. Then there exist words~$\ww_1 \wdots \ww_\rr, \wws$ such that $\ww \gozotR \ww_1 \pdots \ww_\rr \wws$ holds, $\ww_\ii$ lies in~$\WWW{\SS_{\oo{\gg_\ii}}}$ and represents~$\gg_\ii$ for $\ii = 1 \wdots \rr$, and $\wws$ lies in~$\WWW{\SS_{1,2}}$ and represents~$\ggs$.
\end{lemm}

\begin{proof}
As $\SS$ is the union of~$\SS_1$ and~$\SS_2$, every word of~$\WWW\SS$ can be expressed as a concatenation of words in~$\WWW{\SS_1}$ and in~$\WWW{\SS_2}$, say $\ww_1 \pdots \ww_\pp$ with $\ww_\ii \in \WWW{\SS_{\ie_\ii}}$ and $\ie_\ii \not= \ie_{\ii+1}$. For~$\ww$ in~$\WWW\SS$, we denote by~$\AN\ww$ the minimal value of~$\pp$ for which such a decomposition of~$\ww$ exists. Then we prove the result using induction on~$\AN\ww$.

Assume first $\AN\ww = 1$. Then $\ww$ belongs to~$\WWW{\SS_\ie}$ for $\ie = 1$ or~$2$. Two subcases are possible. Assume first $\cl\ww \in \GG_{1,2}$. Then we have $\rr = 0$, with $\AD{\cl\ww} = (\gg)$. By Lemma~\ref{L:Subgroup}\ITEM2, there exists~$\wws$ in~$\WWW{\SS_{1,2}}$ satisfying $\ww \gozotRo \wws$. A fortiori, we have $\ww \gozotR \wws$, whence $\cl{\wws} = \cl{\ww}$, and $\cl{\wws} = \gg = \ggs$, as expected. Assume now $\cl\ww \notin \GG_{1,2}$. Then, we have $\rr = 1$. By~$\PPPPe{\SS_\ie}{\SS_{1,2}}\ww$, there exists~$\ww_1$ in~$\WWW{\SS_{\ie}}$ and $\wws$ in~$\WWW{\SS_{1,2}}$ satisfying $\ww \gozotRe \ww_1 \wws$ and $\cl{\ww_1} \in \TTT_\ie(\SS_{1,2})$. A fortiori, we have $\ww \gozotR \ww_1 \wws$, whence $\cl{\ww_1 \wws} = \cl{\ww}$. As, by construction, $\cl{\wws}$ belongs to~$\GG_{1,2}$, the uniqueness of the amalgam decomposition implies $\cl{\ww_1} = \gg_1$ and $\cl{\wws} = \ggs$, as expected again.

Assume now $\AN\ww \ge 2$. Write $\ww = \vv \uu$ with $\AN\vv = \AN\ww-1$ and $\AN\uu = 1$. Then $\uu$ belongs to~$\WWW{\SS_\ie}$ for $\ie = 1$ or~$2$. Let $(\ff_1 \wdots \ff_\qq, \ffs)$ be the amalgam decomposition of~$\cl\vv$. By induction hypothesis, there exist words $\vv_1 \wdots \vv_\qq, \vvs$ such that $\vv \gozotR \vv_1 \pdots \vv_\qq \vvs$ holds, $\vv_\ii$ lies in~$\WWW{\SS_{\oo{\ff_\ii}}}$ and represents~$\ff_\ii$ for $\ii = 1 \wdots \qq$, and $\vvs$ lies in~$\WWW{\SS_{1,2}}$ and represents~$\ffs$. We consider the various ways $\AD{\cl\ww}$ can be obtained from~$\AD{\cl\vv}$ and from~$\cl\uu$, hereafter denoted by~$\hh$, an element of~$\GG_\ie$.

We consider $\ffs \hh$, which, by hypothesis, lies in~$\GG_\ie$, and its word representative $\vvs \uu$, which lies in~$\WWW{\SS_\ie}$ and satisfies $\AN{\vvs \uu} = 1$. Assume first $\hh \in \GG_{1,2}$. In this case, $\ffs \hh$ lies in~$\GG_{1,2}$ as well, and we have $\rr = \qq$ with $\AD{\cl\ww} = (\ff_1 \wdots \ff_\qq, \ffs \hh)$. By induction hypothesis (or by Lemma~\ref{L:Subgroup}\ITEM2 applied in~$\GR{\SS_\ie}{\RR_\ie}$), there exist a word~$\uu_*$ in~$\WWW{\SS_{1,2}}$ such that $\vvs \uu \gozotR \uu_*$ holds and, therefore, $\uu_*$ represents~$\ffs \uu$. Then we have $\ww \gozotR \vv_1 \pdots \vv_\qq \uu_*$, and the words $\vv_1 \wdots \vv_\qq, \uu_*$ have the expected properties. 

Assume now $\hh \in \GG_\ie \setminus \GG_{1,2}$ with $\qq = 0$ or $\oo{\ff_\qq} \not= \ie$. In this case, $\ffs \hh$ lies in~$\GG_\ie \setminus \GG_{1,2}$, and we have $\rr = \qq+1$ with $\AD{\cl\ww} = (\ff_1 \wdots \ff_\qq, \hh_1, \hh_*)$ where $(\hh_1, \hh_*)$ is the amalgam decomposition of~$\ffs \hh$. Now, as above, we have $\AN{\vvs \uu} = 1$, so, by induction hypothesis, there exist words~$\uu_1, \uu_*$ such that $\vvs \uu \gozotRe \uu_1 \uu_*$ holds, $\uu_1$ lies in~$\WWW{\SS_\ie}$ and represents~$\hh_1$, and $\uu_*$ lies in~$\WWW{\SS_{1,2}}$ and represents~$\hh_*$. Then we have $\ww \gozotR \vv_1 \pdots \vv_\qq \vvs \uu \gozotR \vv_1 \pdots \vv_\qq \uu_1 \uu_*$, and the words~$\vv_1 \wdots \vv_\qq, \uu_1, \uu_*$ have the expected properties.

Assume now $\hh \in \GG_\ie \setminus \GG_{1,2}$ with $\qq \ge 1$ and $\oo{\ff_\qq} = \ie$. Then we consider $\ff_\qq \ffs \hh$, which, by hypothesis, lies in~$\GG_\ie$, and its word representative $\vv_\qq \vvs \uu$: by hypothesis, $\vv_\qq$ lies in~$\WWW{\SS_{\oo{\ff_\qq}}}$, which is~$\WWW{\SS_\ie}$, so $\vv_\qq \vvs \uu$ lies in~$\WWW{\SS_\ie}$ and we have $\AN{\vv_\qq \vvs \uu} = 1$. As above in the case of~$\vvs \uu$, two cases arise. Assume first $\ff_\qq \ffs \hh \in \GG_{1,2}$. In this case, we have $\rr = \qq - 1$ and $\AD\gg = (\ff_1 \wdots \ff_{\qq-1}, \ff_\qq \ffs \hh)$. By induction hypothesis (or by Lemma~\ref{L:Subgroup}), there exists~$\uu_*$ of~$\WWW{\SS_{1,2}}$ satisfying $\vv_\qq \vvs \uu \gozotR \uu_*$ and, therefore, $\uu_*$ represents~$\ff_\qq \ffs \uu$. Then we have $\ww \gozotR \vv_1 \pdots \vv_\qq \vvs \uu \gozotR \vv_1 \pdots \vv_{\qq-1} \uu_*$, and $\vv_1 \wdots \vv_{\qq-1}, \uu_*$ have the expected properties.

Assume finally $\ff_\qq \ffs \hh \in \GG_\ie \setminus \GG_{1,2}$. Then we have $\rr = \qq$ with $\AD\gg = (\ff_1 \wdots \ff_{\qq-1}, \hh_1, \hh_*)$ where $(\hh_1, \hh_*)$ is the amalgam decomposition of~$\ff_\qq \ffs \hh$. As above, we have $\AN{\vv_\qq \vvs \uu} =\nobreak 1$ so, by induction hypothesis, there exist $\uu_1$ and~$\uu_*$ such that $\vvs \uu \gozotR \uu_1 \uu_*$ holds, $\uu_1$ lies in~$\WWW{\SS_\ie}$ and represents~$\hh_1$, and $\uu_*$ lies in~$\WWW{\SS_{1,2}}$ and represents~$\hh_*$. Then we have $\ww \gozotR \vv_1 \pdots \vv_\qq \vvs \uu \gozotR \vv_1 \pdots \vv_{\qq-1} \uu_1 \uu_*$, and $\vv_1 \wdots \vv_{\qq-1}, \uu_1, \uu_*$ have the expected properties.
\end{proof}

\begin{lemm}
\label{L:CanDec}
Assume that $\SS'$ is included in~$\SS$ and $\gg$ lies in~$\SG{\SS'}$. Then all entries of the amalgam decomposition of~$\gg$ lie in~$\SG{\SS'}$. More precisely, if $\AD\gg$ is $(\gg_1 \wdots \gg_\rr, \ggs)$, then $\gg_\ii$ lies in~$\SG{\SS' \cap \SS_{\oo{\gg_\ii}}}$ for $\ii = 1 \wdots \rr$, and $\ggs$ lies in~$\SG{\SS' \cap \SS_{1,2}}$.
\end{lemm}

\begin{proof}
As $\SS'$ is the union of~$\SS' \cap \SS_1$ and~$\SS' \cap \SS_2$, every element of~$\SG{\GG'}$ can be expressed as a product of elements of~$\SG{\SS' \cap \SS_1}$ and~$\SG{\SS' \cap \SS_2}$, say $\gg_1 \pdots \gg_\pp$ with $\gg_\ii \in \SG{\SS' \cap \SS_{\ie_\ii}}$ and $\ie_\ii \not= \ie_{\ii+1}$. For~$\gg$ in~$\SG{\SS'}$, we define~$\ALL{\SS'}{\gg}$ to be the minimal value of~$\pp$ for which such a decomposition of~$\gg$ exists. Then we prove the result using induction on~$\ALL{\SS'}\gg$.

Assume first $\ALL{\SS'}\gg = 1$. Then $\gg$ belongs to~$\GG_\ie$ for~$\ie = 1$ or~$2$, and $\AL\gg \le 1$ holds. If $\gg$ lies in~$\GG_{1,2}$, the amalgam decomposition of~$\gg$ is~$(\gg)$, and $\gg$ lies in~$\SG{\SS'}$. Assume now $\gg \notin \GG_{1,2}$. Let $\ww$ be a word of~$\WWW\SS$ that represents~$\gg$. By $\PPPPe{\SS' \cap \SS_\ie}{\SS_{1,2}}\ww$, we find $\ww_1$ in~$\WWW{\SS' \cap \SS_\ie}$ and $\wws$ in~$\WWW{\SS' \cap \SS_{1,2}}$ satisfying $\ww \gozotRe \ww_1 \wws$ and $\cl{\ww_1} \in \TTT_\ie(\SS_{1,2})$. By uniqueness, the amalgam decomposition of~$\gg$ must be $(\cl{\ww_1}, \cl{\wws})$ and, by construction, $\cl{\ww_1}$ lies in~$\SG{\SS' \cap \SS_\ie}$ and $\cl{\wws}$ lies in~$\SG{\SS' \cap \SS_{1,2}}$.

Assume now $\ALL{\SS'}\gg \ge 2$. Write $\gg = \ff \hh$ with~$\ff, \hh$ in~$\SG{\SS'}$ satisfying $\ALL{\SS'}\ff < \ALL{\SS'}\gg$ and $\ALL{\SS'}\hh = 1$. Let $(\ff_1 \wdots \ff_\qq, \ffs)$ be the amalgam decomposition of~$\ff$. By induction hypothesis, $\ff_1 \wdots \ff_\qq$ and~$\ffs$ lie in~$\SG{\SS'}$.
Then, according to the scheme used in the proof of Lemma~\ref{L:Pieces}, the amalgam decomposition of~$\gg$ is of one of the types

\ITEM1\quad $(\ff_1 \wdots \ff_\qq, \ffs \hh)$,

\ITEM2\quad $(\ff_1 \wdots \ff_\qq, \hh_1, \hh_*)$ \quad with $\ALL{\SS'}{\ffs \hh} = 1$ and $(\hh_1, \hh_*) = \AD{\ffs \hh}$,

\ITEM3\quad $(\ff_1 \wdots \ff_{\qq-1}, \ff_\qq \ffs \hh)$, 

\ITEM4\quad $(\ff_1 \wdots \ff_{\qq-1}, \hh_1, \hh_*)$ \quad with $\ALL{\SS'}{\ff_\qq \ffs \hh} = 1$ and $(\hh_1, \hh_*) = \AD{\ff_\qq \ffs \hh}$.

\noindent By hypothesis, all involved elements, in particular $\hh_1$ and~$\hh_*$ in cases~\ITEM2 and~\ITEM4 owing to the induction hypothesis, lie in~$\SG{\SS' \cap \SS_1}$, $\SG{\SS' \cap \SS_2}$, or~$\SG{\SS' \cap \SS_{1,2}}$ as expected.
\end{proof}

Using Lemmas~\ref{L:CanDec} and~\ref{L:Subgroup}\ITEM2, we can obtain a more precise version of Lemma~\ref{L:Pieces}.

\begin{lemm}
\label{L:PiecesBis} 
Assume that $\SS'$ is included in~$\SS$ and $\ww$ is a word of~$\WWW\SS$ that represents an element of~$\SG{\SS'}$. Let $(\gg_1 \wdots \gg_\rr, \ggs)$ be the amalgam decomposition of~$\cl\ww$. Then there exist words $\ww'_1 \wdots \ww'_\rr, \wws'$ such that $\ww \gozotR \ww'_1 \pdots \ww'_\rr \wws'$ holds, $\ww'_\ii$ lies in~$\WWW{\SS' \cap \SS_{\oo{\gg_\ii}}}$ and represents~$\gg_\ii$ for $\ii = 1 \wdots \rr$, and $\wws'$ lies in~$\WWW{\SS' \cap \SS_{1,2}}$ and represents~$\ggs$.
\end{lemm}

The difference with Lemma~\ref{L:Pieces} is that, here, we assume that the words~$\ww'_\ii$ and~$\wws'$ lie in~$\WWW{\SS'}$, that is, involve letters from~$\SS'$ only.

\begin{proof}
By Lemma~\ref{L:Pieces}, we first find $\ww_1 \wdots \ww_\rr, \wws$ that represent~$\gg_1 \wdots \gg_\rr, \ggs$. Now, by Lemma~\ref{L:CanDec}, the elements $\gg_1 \wdots \gg_\rr, \ggs$ lie in~$\SG{\SS'}$. Then Lemma~\ref{L:Subgroup}\ITEM2 implies the existence of $\ww'_1 \wdots \ww'_\rr, \wws'$ that satisfy $\ww_\ii \gozotR \ww'_\ii$ for $\ii = 1 \wdots \rr$ and $\wws \gozotR \wws'$, and all lie in~$\WWW{\SS'}$ in addition to the other properties. Then we have $\ww \gozotR \ww_1 \pdots \ww_\rr \wws \gozotR \ww'_1 \pdots \ww'_\rr \wws'$.
\end{proof}

Our aim is to prove that $(\SS, \RR)$ satisfies Property~$\Propp$ and, therefore, we must construct, for every subset~$\SS_0$ of~$\SS$, a distinguished $\SG{\SS_0}$-transversal~$\TTT(\SS_0)$ in the group~$\GG$. To this end, we use the transversals~$\TTT_1(\SS_0 \cap \SS_1)$ and~$\TTT_2(\SS_0 \cap \SS_2)$. The natural idea is that $\TTT(\SS_0)$ consists of those elements of~$\GG$ whose tail belongs to~$\TTT_\ie(\SS_0 \cap \SS_\ie)$ for~$\ie = 1$ or~$2$---we recall from Definition~\ref{D:CanDec} that the tail~$\tail\gg$ of~$\gg$ is the maximal final chunk of~$\gg$ that lies in~$\GG_1$ or~$\GG_2$. To avoid the ambiguities resulting from the fact that, on~$\GG_1 \cap \GG_2$, the transversals $\TTT_1(\SS_0 \cap \SS_1)$ and $\TTT_2(\SS_0 \cap \SS_2)$ need not coincide, we systematically choose $\TTT_1(\SS_0 \cap \SS_1)$ whenever possible. 

\begin{defi}
For~$\SS_0$ included in~$\SS$, we put 
\begin{equation}
\TTT(\SS_0) = \{ \gg \in \GG \mid \tail\gg \in \TTT_1(\SS_0 \cap \SS_1) \cup (\TTT_2(\SS_0 \cap \SS_2) \setminus \GG_1) \}.
\end{equation}
\end{defi}

We shall show that $\TTT$ is an $\SS$-sequence of transversals witnessing for Property~$\Propp$. Note that, the hypothesis that $\TTT_1(\SS_0 \cap \SS_1)$ contains~$1$ implies that $\TTT(\SS_0)$ contains~$1$ as well. We begin with two preparatory results.

\begin{lemm}
\label{L:Length}
Assume that $\SS_0$ is included in~$\SS$, that $\gg$ lies in~$\TTT(\SS_0)$ and that $\gg'$ lies in the left-$\SG{\SS_0}$-coset of~$\gg$. Let $(\gg_1 \wdots \gg_\rr, \ggs)$ and $(\gg'_1 \wdots \gg'_{\rr'}, \gg'_*)$ be the amalgam decompositions of~$\gg$ and~$\gg'$. Then we have \ITEM1 $\rr' \ge \rr$, \ITEM2 $\oo{\gg'_\rr} = \oo{\gg_\rr}$ if $\rr \ge 1$ holds, and \ITEM3 $\gg'_\ii = \gg_\ii$ for~$1 \le \ii < \rr$.
\end{lemm} 

\begin{proof}
By hypothesis, there exists~$\hh$ in~$\SG{\SS_0}$ such that $\gg' = \gg \hh$ holds. Let $(\hh_1 \wdots \hh_\qq, \hh_*)$ be the amalgam decomposition of~$\hh$. By Lemma~\ref{L:CanDec} applied to~$\SG{\SS_0}$, the elements~$\hh_1 \wdots \hh_\qq, \hh_*$ all lie in~$\SG{\SS_0}$. The results are obvious or vacuously true for $\rr = 0$, and we assume $\rr \ge 1$, with $\oo{\gg_\rr} = \ie$. We determine~$\AD{\gg'}$ starting from~$\AD\gg$ and~$\AD\hh$ using arguments similar to those of Lemma~\ref{L:Pieces}. The point is that the case when the length of the amalgam decomposition decreases turns out to be impossible. Write $\overline\ie$ for~$3 - \ie$. 

For $\qq = 0$, we find $\AD{\gg'} = (\gg_1 \wdots \gg_\rr, \hh \hh_*)$, and $\AL{\gg'} = \rr$, so everything is clear.

Next assume $\qq \ge 1$ with $\hh_1 \in \TTT_{\overline\ie}(\SS_{1,2})$. Assume $\AD{\ggs \hh} = (\ff_1 \wdots \ff_\pp, \ffs)$. By hypothesis, $\ggs \hh_1$ lies in $\GG_{\overline\ie} \setminus \GG_{1,2}$, which, by uniqueness of the amalgam decomposition, implies $\ff_1 \in \TTT_{\overline\ie}(\SS_{1,2})$ (and $\pp = \qq$, but this is not needed here). In this case, we have $\AD{\gg \gg_0} = (\gg_1 \wdots \gg_\rr, \ff_1 \wdots \ff_\pp, \ffs)$. We deduce $\AL{\gg'} = \rr + \pp \ge \rr$ and the other expected results. 

Assume now $\qq \ge 1$ with $\hh_1 \in \TTT_\ie(\SS_{1,2})$ and $\gg_\rr \ggs \hh_1 \notin \GG_{1,2}$. Let $\AD{\gg_\rr \ggs \hh} = (\ff_1 \wdots \ff_\pp, \ffs)$. By hypothesis, $\gg_\rr \ggs \hh_1$ lies in $\GG_\ie \setminus \GG_{1,2}$, which implies $\pp \ge 1$ and $\ff_1 \in \TTT_\ie(\SS_{1,2})$. In this case, we find $\AD{\gg'} = (\gg_1 \wdots \gg_{\rr-1}, \ff_1 \wdots \ff_\pp, \ffs)$. We deduce $\AL{\gg'} = \rr-1 + \pp \ge \rr$ and the other expected results. 

Assume finally $\qq \ge 1$ with $\hh_1 \in \TTT_\ie(\SS_{1,2})$ and $\gg_\rr \ggs \hh_1 \in \GG_{1,2}$. By definition, the hypothesis that $\gg$ belongs to~$\TTT(\SS_0)$ implies that $\gg_\rr \ggs$ lies in~$\TTT_\ie(\SS_0 \cap \SS_\ie)$. On the other hand, by Lemma~\ref{L:CanDec} applied with $\SS' = \SS_0$, the hypothesis that $\hh$ lies in~$\SG{\SS_0}$ implies that $\hh_1$ lies in~$\SG{\SS_0 \cap \SS_\ie}$. Here comes the point. Let $\ww$ be any word of~$\WWW{\SS_\ie}$ that represents $\gg_\rr \ggs \hh_1$. Then $\PPPPe{\SS_{1,2}}{\SS_0 \cap \SS_\ie}\ww$ implies the existence of~$\vv$ in~$\WWW{\SS_{1,2}}$ and~$\uu$ in~$\WWW{\SS_{1,2} \cap \SS_0}$ satisfying $\ww \gozotRe \vv \uu$ and $\cl\vv \in \TTT_\ie(\SS_0 \cap \SS_\ie)$. Now, as $\hh_1$ lies in~$\SG{\SS_0 \cap \SS_\ie}$ and $\TTT(\SS_0 \cap \SS_\ie)$ is a $\SG{\SS_0 \cap \SS_\ie}$-transversal, we must have $\cl\vv = \gg_\rr \ggs$ and $\cl\uu = \hh_1$. But this implies $\gg_\rr \ggs \in \GG_{1,2}$, whence $\gg_\rr \in \GG_{1,2}$, which contradicts the definition of~$\AD\gg$. This case, which is the one that could possibly lead to $\AL{\gg'} < \AL{\gg}$ is therefore impossible. So the proof is complete.
\end{proof}

\begin{lemm}
\label{L:Transversal}
For every~$\SS_0$ included in~$\SS$, the set $\TTT(\SS_0)$ does not contain two distinct elements lying in the same left-$\SG{\SS_0}$-coset.
\end{lemm}

\begin{proof} 
Assume that $\gg, \gg'$ belong to~$\TTT(\SS_0)$ and lie in the same left-$\SG{\SS_0}$-coset of~$\GG$. By Lemma~\ref{L:Length}, we have $\AL\gg' \ge \AL\gg$, and $\AL\gg \ge \AL{\gg'}$, whence $\AL\gg = \AL{\gg'}$. Put $\hh = \gg\inv \gg'$, an element of~$\SG{\SS_0}$. We now prove $\gg = \gg'$ by considering the various possible values of~$\AL\gg$. 

Assume first $\AL\gg = 0$. Then, by definition, $\gg$ and~$\gg'$ belong to~$\TTT_1(\SS_0 \cap \SS_1)$ (and to~$\GG_{1,2}$), hence $\hh$ belongs to~$\SG{\SS_0} \cap \GG_1$, which, by Lemma~\ref{L:Subgroup}, is~$\SG{\SS_0 \cap \SS_1}$. So $\gg$ and $\gg'$ belong to the same left-$\SG{\SS_0 \cap \SS_1}$-coset in~$\GG_1$. Then, the hypothesis that $\TTT_1(\SS_0 \cap \SS_1)$ is a $\SG{\SS_0 \cap \SS_1}$-transversal in~$\GG_1$ implies~$\gg = \gg'$.

Assume now $\AL\gg = 1$. By definition, there exist~$\ie$ and~$\ie'$ such that $\gg$ belongs to~$\TTT_\ie(\SS_0 \cap \SS_\ie)$ and $\gg'$ belongs to~$\TTT_{\ie'}(\SS_0 \cap \SS_{\ie'})$. By Lemma~\ref{L:Length}\ITEM2, we must have $\ie = \ie'$. It follows that $\hh$ belongs to~$\SG{\SS_0} \cap \GG_\ie$, which is~$\SG{\SS_0 \cap \SS_\ie}$. So $\gg$ and $\gg'$ belong to the same left-$\SG{\SS_0 \cap \SS_\ie}$-coset in~$\GG_\ie$ and, as $\TTT_\ie(\SS_0 \cap \SS_\ie)$ is a $\SG{\SS_0 \cap \SS_\ie}$-transversal in~$\GG_\ie$, we must have $\gg = \gg'$.

Assume finally $\rr \ge 2$. By Lemma~\ref{L:Length}\ITEM3, the amalgam decompositions of~$\gg$ and~$\gg'$ have the form $\AD\gg = (\gg_1 \wdots \gg_\rr, \ggs)$ and $\AD{\gg'} = (\gg_1 \wdots \gg_{\rr-1}, \gg'_\rr, \gg'_*)$. By definition of~$\TTT$, the elements~$\tail\gg$ and~$\tail{\gg'}$ belong to~$\TTT(\SS_0)$ since $\gg$ and~$\gg'$ do, and $\gg' = \gg \hh$ implies
$$\tail{\gg'} = \gg'_\rr \gg'_* = \gg_1\inv \pdots \gg_{\rr-1}\inv \gg' = \gg_1\inv \pdots \gg_{\rr-1}\inv \gg \hh = \gg_\rr \ggs \hh = \tail\gg \hh,$$
so $\tail{\gg'}$ and~$\tail\gg$ lie in the same~$\SG{\SS_0}$-coset. As, by construction, we have $\AL{\tail\gg} = \AL{\tail{\gg'}} = 1$, the induction hypothesis implies $\tail\gg = \tail{\gg'}$, whence $\gg = \gg'$.
\end{proof}
 
We are now ready to complete the argument.

\begin{proof}[Proof of Proposition~\ref{P:Union}]
Let $\SS', \SS_0$ be included in~$\SS$. Two points have to be established, namely that $\TTT(\SS_0)$ is a $\SG{\SS_0}$-transversal in~$\GG$ and that $\PPP_{\TTT}(\SS', \SS_0, \ww)$ holds for every word~$\ww$ of~$\WWW\SS$ that represents an element of~$\SG{\SS'}$. We begin with the second point, which we establish using induction on~$\AL{\cl\ww}$. 

\ITEM1 Assume first $\AL{\cl\ww} \le 1$ with $\cl\ww \in \GG_1$. Applying Lemma~\ref{L:PiecesBis} to~$\SS'$ and~$\ww$, we first find~$\ww'$ in~$\WWW{\SS' \cap \SS_1}$ satisfying $\ww \gozotR \ww'$. Then, applying~$\PPPPo{\SS' \cap \SS_1}{\SS' \cap \SS_0 \cap \SS_1}{\ww'}$, we find~$\vv$ in~$\WWW{\SS' \cap \SS_1}$ and $\uu$ in~$\WWW{\SS' \cap \SS_0 \cap \SS_1}$ satisfying $\ww' \gozotRo \vv \uu$ and $\cl\vv \in \TTT_1(\SS_0 \cap \SS_1)$. By composing, we deduce $\ww \gozotR \vv \uu$ and, by definition, $\TTT_1(\SS_0 \cap \SS_1)$ is included in~$\TTT(\SS_0)$, so $\vv$ and~$\uu$ witness for~$\PPP_{\TTT}(\SS', \SS_0, \ww)$.

\ITEM2 Assume now $\AL{\cl\ww} \le 1$ with $\cl\ww \in \GG_2 \setminus \GG_1$. Arguing as above, but now in~$\GG_2$, we find~$\vv$ in~$\WWW{\SS' \cap \SS_2}$ and $\uu$ in~$\WWW{\SS' \cap \SS_0 \cap \SS_2}$ satisfying $\ww \gozotRt \vv \uu$ and $\cl\vv \in \TTT_2(\SS_0 \cap \SS_2)$. By definition, $\TTT_2(\SS_0 \cap \SS_2) \setminus \GG_1$ is included in~$\TTT(\SS_0)$, so, if $\cl\vv$ does not lie in~$\GG_1$, the words~$\vv$ and~$\uu$ witness for~$\PPP_{\TTT}(\SS', \SS_0, \ww)$. If $\cl\vv$ lies in~$\GG_1$, hence in~$\GG_{1,2}$, one more step is needed to go from~$\TTT_2(\SS_0 \cap \SS_2)$ to~$\TTT_1(\SS_0 \cap \SS_1)$. Now, we have $\AL{\cl\uu} = 0$ with $\cl\uu \in \GG_1$, so by~\ITEM1 we can find~$\vv'$ in~$\WWW{\SS' \cap \SS_1}$ and $\uu'$ in~$\WWW{\SS' \cap \SS_0 \cap \SS_1}$ satisfying $\vv \gozotRo \vv' \uu'$ and $\cl{\vv'} \in \TTT_1(\SS_0 \cap \SS_1)$, hence $\cl{\vv'} \in \TTT(\SS_0)$. Then we deduce $\ww \gozotR \vv' (\uu' \uu)$, and the words~$\vv'$ and~$\uu' \uu$ witness for~$\PPP_{\TTT}(\SS', \SS_0, \ww)$.

\ITEM3 Assume now $\AL{\cl\ww} \ge 2$. Let $\AD{\cl\ww} = (\gg_1 \wdots \gg_\rr, \ggs )$. Applying Lemma~\ref{L:PiecesBis} to~$\SS'$ and~$\ww$, we find words~$\ww_1 \wdots \ww_\rr, \wws$ such that $\ww \gozotR \ww_1 \pdots \ww_\rr \wws$ holds, $\ww_\ii$ lies in~$\WWW{\SS' \cap \SS_{\oo{\gg_\ii}}}$ and represents~$\gg_\ii$ for $\ii = 1 \wdots \rr$, and $\wws$ lies in~$\WWW{\SS' \cap \SS_{1,2}}$ and represents~$\ggs$. Let $\ie = \oo{\gg_\rr}$. By construction, we have $\AL{\gg_\rr \ggs} = 1$ with~$\gg_\rr \ggs \in \GG_\ie$. By~\ITEM1 or~\ITEM2 applied to~$\ww_\rr \wws$, we find words~$\vv_\rr$ in~$\WWW{\SS' \cap \SS_\ie}$ and~$\uu$ in~$\WWW{\SS' \cap \SS_0 \cap \SS_\ie}$ satisfying $\ww_\rr \wws \gozotRe \vv_\rr \uu$ and $\cl{\vv_\rr} \in \TTT_\ie(\SS_0 \cap \SS_\ie)$. So, if we put $\vv = \ww_1 \pdots \ww_{\rr-1} \vv_\rr$, then $\ww \gozot \vv \uu$ is true. 

Two cases are possible. If $\cl{\vv_\rr}$ does not belong to~$\GG_{1,2}$, then $(\gg_1 \wdots \gg_{\rr-1}, \cl{\vv_\rr})$ is the amalgam decomposition of~$\cl\vv$, and $\cl\vv$ lies in~$\TTT(\SS_0)$. Then the words~$\vv$ and~$\uu$ witness for~$\PPP_{\TTT}(\SS', \SS_0, \ww)$. On the other hand, if $\cl{\vv_\rr}$ belongs to~$\GG_{1,2}$, we have $\AL{\cl\vv} < \AL{\cl\ww}$. Then, by induction hypothesis, we can find~$\vv'$ in~$\WWW{\SS'}$ and~$\uu'$ in~$\WWW{\SS' \cap \SS_0}$ satisfying $\vv \gozotR \vv' \uu'$ and $\cl{\vv'} \in \TTT(\SS_0)$. Then we deduce $\ww \gozotR \vv' (\uu' \uu)$, and the words~$\vv'$ and~$\uu' \uu$ witness for~$\PPP_{\TTT}(\SS', \SS_0, \ww)$.

So, the proof of~$\PPP_{\TTT}(\SS', \SS_0, \ww)$ is complete, and the only thing still to be proven is that $\TTT(\SS_0)$ is an $\SG{\SS_0}$-transversal. This amonts to proving two facts, namely that two distinct elements of~$\TTT(\SS_0)$ never lie in the same $\SG{\SS_0}$-coset, which has been done in Lemma~\ref{L:Transversal}, and that, for every~$\gg$ in~$\GG$, there exists an element of~$\TTT(\SS_0)$ is the left-$\SG{\SS_0}$-coset of~$\gg$. Now the latter point follows from~$\PPP_{\TTT}(\SS, \SS_0, \ww)$. Indeed, let~$\ww$ be any word representing~$\gg$. By~$\PPP_{\TTT}(\SS, \SS_0, \ww)$, we find~$\vv$ in~$\WWW\SS$ and~$\uu$ in~$\WWW{\SS_0}$ satisfying $\ww \gozotR \vv \uu$ and $\cl\vv \in \TTT(\SS_0)$. Then $\cl\uu$ lies in~$\SG{\SS_0}$, and $\cl\vv$ is an element of~$\TTT(\SS_0)$ lying in the left-$\SG{\SS_0}$-coset of~$\gg$.Ê 
\end{proof}

Thus the proof that all Artin--Tits presentations of type~FC satisfy Properties~$\Propp$ and~$\Prop$ (Proposition~\ref{P:FC}) is complete.

%%%%%%%%

\section{A syntactic approach}
\label{S:RAAG}

Our proof of Property~$\Prop$ for type~FC is essentially based on the existence of a distinguished normal form for the elements of the considered Artin--Tits group. Although the approach is simple, it seems unlikely that it could lead to a proof of Conjecture~\ref{C:Main} as the construction of a similar normal form in the case of general Artin--Tits presentations seems out of reach. Also, the possible interest of Property~$\Prop$ is to provide some control of a group even in the case when no solution to the word problem is known, and using a preexisting solution to that word problem to establish Property~$\Prop$ is not natural. In this section, we develop another approach, which leads to Property~$\Prop$ directly and does not appeal to any normal form or, more generally, to any solution of the word problem. The principle is to start from a derivation that involves transformations of type~$0, 1$, and~$\infty$ and eliminate type~$\infty$ transformations by introducing type~$2$ transformations instead. 

We show in Subsection~\ref{SS:RAAG} how to realize this strategy in the case of right-angled Artin--Tits presentations. A few additional comments are included in Subsection~\ref{SS:Further}. 

%%%%

\subsection{Right-angled Artin--Tits presentations}
\label{SS:RAAG}

An Artin--Tits presentation is said to be \emph{right-angled} if it only involves commutation relations $\ss \tt = \tt \ss$. Right-angled Artin--Tits presentations all are of FC type, so Proposition~\ref{P:FCMain} implies 

\begin{prop}
\label{P:RA}
Every right-angled Artin--Tits presentation satisfies Property~$\Prop$.
\end{prop}

We shall now give a new proof of this result relying on a purely syntactic approach. The argument requires several auxiliary developments and will be completed at the end of the subsection. Until there, we assume that $(\SS, \RR)$ is a fixed right-angled Artin--Tits presentation, and denote by~$\GG$ the group~$\GR\SS\RR$. 

Assume that $\ww$ is a word of~$\WWW\SS$ that represents~$1$ in~$\GG$. Then we have $\ww \gozoiR \ew$, so there exists a sequence of words $\ww_0 \wdots \ww_\mm$ with $\ww_0 = \ww$, $\ww_\mm = \ww'$ and, for every~$\kk$, the word~$\ww_\kk$ is obtained from~$\ww_{\kk-1}$ by applying a special transformation of type~$0$, $1$, or~$\infty$. The idea will be to mark the additional letters created in type~$\infty$ steps and to follow their subsequent destiny in the derivation. The point is that, when we erase these additional letters in a convenient way, the remaining words will make a derivation in terms of type~$0, 1$, and~$2$ transformations.

Implementing the above simple principle requires taking into account the step at which each additional letter is created. This leads us to introducing an augmented alphabet~$\aS$ where all letters include a number that corresponds to their age, that is, the date of their creation and, accordingly, an augmented version~$\agoesR{}$ of the relation~$\goesR{}$ for words of~$\WWW\aS$.

\begin{defi}
We denote by~$\aS$ the alphabet consisting of one letter~$\as\ii$ for each letter~$\ss$ of~$\SS$ and each nonnegative integer~$\ii$; we identify~$\ss$ with~$\as0$. We denote by~$\phi$ the alphabetical homomorphism of~$\WWW{\aS}$ to~$ \WWW\SS$ that maps every letter~$\as\ii$ to~$\ss$. For~$\hh \ge 0$, we denote by~$\pi_\hh$ the alphabetical homomorphism of~$\WWW{\aS}$ to itself that maps~$\as\ii$ to itself for~$\ii < \hh$ and to~$\ew$ for~$\ii \ge \hh$.
\end{defi}

Words of~$\WWW{\aS}$ will be called \emph{augmented}. Thus, assuming that $\tta, \ttb, \ttc$ are letters of~$\SS$, a typical augmented word is $\atta1 \attB0 \attc2 \attA1 \attc0$. Its image under~$\phi$ is the (standard) word~$\tta \ttB \ttc \ttA \ttc$ (all indices are forgotten), whereas its image under~$\pi_2$ is $\atta1 \attB0 \attA1 \attc0$ (all letters with index $2$ and larger are erased), and its image under~$\pi_1$ is~$\attB0 \attc0$.

\begin{defi}
\emph{Special} transformations on augmented words are defined as follows:

- type~$0$: Remove some pair $\as\ii^\ee \as\jj^{-\ee}$, $\ee = \pm1$, and replace each remaining occurrence of~$\as\ii^{\pm1}$ and~$\as\jj^{\pm1}$ with~$\as{\min(\ii, \jj)}^{\pm1}$;

- type~$1$: Replace $\as\ii^\ee \at\jj^\ee$ with $\at\jj^\ee \as\ii^\ee$, where $\ss \tt = \tt \ss$ is a relation of~$\RR$ and $\ee$ is~$\pm1$;

- type~$2$: Replace $\as\ii^\ee \at\jj^{-\ee}$ with $\at\jj^{-\ee} \as\ii^\ee$, where $\ss \tt = \tt \ss$ is a relation of~$\RR$ and $\ee$ is~$\pm1$;

- type~$\infty$: Insert some pair $\as\ii^\ee \as\ii^{-\ee}$, $\ee = \pm1$, where $\ii$ is strictly larger than every index of letter occurring in the current word.

\noindent 
We write $\ww \agoesR{X} \ww'$ if $\ww'$ can be obtained from~$\ww$ by a finite sequence of special transformations of type belonging to~$X$.
\end{defi}

Special transformations on augmented words are liftings of the corresponding special transformations on words: a straightforward inspection shows that, if an augmented word~$\ww'$ is the image of another augmented word~$\ww$ under a transformation of type~$\ii$, then the (non-augmented) word~$\phi(\ww')$ is the image of~$\phi(\ww)$ under a transformation of type~$\ii$.

The first observation is that every sequence of special transformations on standard words can be lifted into a sequence of special transformations on augmented words. Hereafter, we say that $(\ww_0, \ww_1, ...)$ is a \emph{special} sequence of (possibly augmented) words if every word~$\ww_\kk$ is obtained from~$\ww_{\kk-1}$ using a special transformation.

\begin{lemm}
\label{L:Lift}
For every special sequence $(\ww_0, \ww_1, ...)$ in~$ \WWW\SS$, there exists a special sequence of augmented words $(\tw_0, \tw_1, ...)$ such that $\tw_0$ equals~$\ww_0$ and, for every~$\kk$, the word~$\ww_\kk$ is the image of~$\tw_\kk$ under~$\phi$ and, moreover, the type of the transformation connecting~$\tw_{\kk-1}$ to~$\tw_\kk$ is the same as the type of the transformation connecting~$\ww_{\kk-1}$ to~$\ww_\kk$.
\end{lemm}

\begin{proof}
We start from $\tw_0 = \ww_0$ and use an induction. For steps of type~$0, 1$, or~$2$, we just mimick the transformation. For steps of type~$\infty$, that is, when $\ww_\kk$ is obtained from~$\ww_{\kk-1}$ by inserting a pair~$\ss^\ee \ss^{-\ee}$, we define~$\tw_\kk$ from~$\tw_{\kk-1}$ by inserting $\as\ii^\ee \as\ii^{-\ee}$, where $\ii$ is the maximum of all indices of letters occurring in~$\tw_\kk$ augmented by~$1$.
\end{proof}

\begin{exam}
\label{X:Abelian2}
Assume $\SS = \{\tta, \ttb, \ttc\}$ and let $\RR = \{\tta \ttb = \ttb \tta, \ttb \ttc = \ttc \ttb, \tta \ttc = \ttc \tta\}$. Then $(\SS, \RR)$ is a right-angled Artin--Tits presentation. As in Section~\ref{S:Prop}, consider $\ww = \tta \ttB \ttc \ttA \ttb \ttC$. Then, starting from the special sequence
\begin{equation*}
\tta \ttB \ttc \ttA \ttb \ttC
\goes\infty \tta \ttB \ttA \tta \ttc \ttA \ttb \ttC
\goesR1 \tta \ttA \ttB \tta \ttc \ttA \ttb \ttC
\goes0 \ttB \tta \ttc \ttA \ttb \ttC
\goesR1 \ttB \ttc \tta \ttA \ttb \ttC
\goes0 \ttB \ttc \ttb \ttC
\goesR1 \ttB \ttb \ttc \ttC
\goes0 \ttc \ttC
\goes0 \ew
\end{equation*}
and using the method of Lemma~\ref{L:Lift} gives
\begin{multline*}
\atta0 \attB0 \attc0 \attA0 \attb0 \attC0
\agoes\infty \atta0 \attB0 \attA1 \atta1 \attc0 \attA0 \attb0 \attC0
\agoesR1 \atta0 \attA1 \attB0 \atta1 \attc0 \attA0 \attb0 \attC0\\
\agoes0 \attB0 \atta0 \attc0 \attA0 \attb0 \attC0
\agoesR1 \attB0 \attc0 \atta0 \attA0 \attb0 \attC0
\agoes0 \attB0 \attc0 \attb0 \attC0
\agoesR1 \attB0 \attb0 \attc0 \attC0
\agoes0 \attc0 \attC0
\agoes0 \ew,
\end{multline*}
where we write $\atta\ii$ for $\atta{[\ii]}$.
\end{exam}

The main point now is to control the properties of the augmented words that appear when special sequences are lifted. 

\begin{defi}
\ITEM1 An augmented word of the form $\ar\hh^\ee \ww \ar\hh^{-\ee}$ is called an \emph{$\hh$-pair}; it is said to be \emph{commutative} if $\rr \ss = \ss \rr$ is a relation of~$\RR$ for every letter~$\ss$ such that $\ss^{\pm1}$ occurs in~$\phi(\pi_\hh(\ww))$ (that is, such that some letter~$\as\ii^{\pm1}$ with $\ii \le \hh$ occurs in~$\ww$).

\ITEM2 An augmented word~$\ww$ is said to be \emph{regular} if, for every~$\hh \ge 1$, either $\ww$ contains no letter of index~$\hh$, or it contains exactly two such letters forming a commutative $\hh$-pair.
\end{defi}

The argument for Proposition~\ref{P:RA} will be inductive. Here is the basic step for the induction.

\begin{lemm}
\label{L:BasicStep}
Assume that $\ww$ is a regular augmented word and $\ww'$ is obtained from~$\ww$ using one special transformation. Then

\ITEM1 The word~$\ww'$ is regular;

\ITEM2 If $\ww$ or $\ww'$ contains at least one letter of positive index and if $\hh$ is the highest index of a letter occurring in~$\ww$ and~$\ww'$, the relation $\pi_\hh(\ww) \agoesRlong{0,1,2, \infty} \pi_\hh(\ww')$ holds;

\ITEM3 The relation $\pi_1(\ww) \gozotR \pi_1(\ww')$ holds.
\end{lemm}

\begin{proof}
We first establish~\ITEM1 and~\ITEM2 by considering the possible types of special transformations connecting~$\ww$ to~$\ww'$. We begin with type~$0$, that is, we assume that $\ww'$ is obtained from~$\ww$ by removing some factor $\as\ii^\ee \as\jj^{-\ee}$, $\ee = \pm1$, and replacing the remaining occurrences of~$\as\ii^{\pm1}$ and~$\as\jj^{\pm1}$ with~$\as{\min(\ii, \jj)}^{\pm1}$. 

\ITEM1 If $\hh$ is neither~$\ii$ nor~$\jj$, the possible letters of index~$\hh$ in~$\ww'$ are those of~$\ww$, and the commutativity condition is unchanged (if the deleted factor is not nested in the possible $\hh$-pair), or weakened. In any case, the hypothesis that $\ww$ satisfies the properties implies that $\ww'$ does too. Assume now $\hh = \ii < \jj$. Then we must be in the case when $\ww$ contains a commutative $\hh$-pair, say $\as\hh^\ee \vv \as\hh^{-\ee}$, and $\ww'$ is obtained by removing one of the letters~$\as\hh^{-\ee}$ together with some adjacent~$\as\jj^{\pm1}$, which occurs in some $\jj$-pair $\as\jj^\dd \vv' \as\jj^{-\dd}$. So, up to symmetries, we may assume that we are in the situation
$$ ... \ \as\hh^\ee \vv \as\hh^{-\ee} \ \as\jj^\ee \vv' \as\jj^{-\ee} \ ... \quad\agoes0\quad ... \ \as\hh^\ee \vv \ \vv' \as\hh^{-\ee}\ ... \, ,$$
and we see that the expected properties hold in~$\ww'$: indeed, $\ww'$ contains one $\hh$-pair, no $\jj$-pair, and the $\hh$-pair is commutative, because, by hypothesis, the letter~$\ss$ commutes with all letters of~$\phi(\pi_\hh(\vv))$ and all letters of~$\phi(\pi_\jj(\vv'))$, hence, a fortiori, with all letters of~$\phi(\pi_\hh(\vv\, \vv'))$. The other cases are similar. Finally, assume $\hh = \ii > \jj$. Then the typical case is
$$ ... \ \as\hh^\ee \vv \as\hh^{-\ee} \ \as\jj^\ee \vv' \as\jj^{-\ee} \ ... \quad\agoes0\quad ... \ \as\jj^\ee \vv \ \vv' \as\jj^{-\ee}\ ... \, ,$$
and, again, the hypothesis that~$\ss$ commutes with all letters of~$\phi(\pi_\hh(\vv))$ and all letters of~$\phi(\pi_\jj(\vv'))$ implies that $\ss$ commutes with all letters of~$\phi(\pi_\jj(\vv\, \vv'))$. So $\ww'$ is regular.

\ITEM2 If $\ii$ and~$\jj$ are strictly smaller than~$\hh$, the letters of index~$\ii$ and~$\jj$ are invariant when $\pi_\hh$ is applied, and the word~$\pi_\hh(\ww')$ is obtained from~$\pi_\hh(\ww)$ by a type~$0$ transformation, that is, $\pi_\hh(\ww) \agoesR{0} \pi_\hh(\ww')$ holds. Assume now $\ii = \hh > \jj$.
Up to symmetries, we are in the situation of
$$ ... \ \as\hh^\ee \vv \as\hh^{-\ee} \ \as\jj^\ee \vv' \as\jj^{-\ee} \ ... \quad\agoes0\quad ... \ \as\jj^\ee \vv \ \vv' \as\jj^{-\ee}\ ... \, .$$
When we apply~$\pi_\hh$, the above words respectively become
$$ ... \ \pi_\hh(\vv) \ \as\jj^\ee \pi_\hh(\vv') \as\jj^{-\ee} \ ... \quad \mbox{and}\quad ... \ \as\jj^\ee \pi_\hh(\vv) \pi_\hh(\vv') \as\jj^{-\ee} \ ... \, .$$
The point is that, as, by hypothesis, the $\hh$-pair $\as\hh^\ee \vv \as\hh^{-\ee}$ is commutative, the letter~$\ss$ commutes with every letter occurring in~$\pi_\hh(\vv)$, and, therefore, we have $\pi_\hh(\vv) \ \as\jj^\ee \agoesR{1,2} \as\jj^\ee \ \pi_\hh(\vv)$, whence $\pi_\hh(\ww) \agoesR{1,2} \pi_\hh(\ww')$.

We turn to type~$1$ and assume now that $\ww'$ is obtained from~$\ww$ by replacing $\as\ii^\ee \at\jj^\ee$ with $\at\jj^\ee \as\ii^\ee$, where $\ss \tt = \tt \ss$ is a relation of~$\RR$ and $\ee$ is~$\pm1$. 

\ITEM1 Here the same letters occur in~$\ww$ and~$\ww'$, so the possible letters of index~$\hh$ occurring in~$\ww$ and~$\ww'$ are the same. In particular, there exists an $\hh$-pair in~$\ww$ if and only if there is one in~$\ww'$. Moreover, the hypothesis that the possible $\hh$-pair is commutative in~$\ww$ implies that the corresponding $\hh$-pair in~$\ww'$ is commutative as well. Indeed, if $\hh$ is neither~$\ii$ nor~$\jj$, the letters nested in the $\hh$-pair are not changed; otherwise, up to symmetries, the typical case is
$$ ... \as\ii^\ee \, \ar\hh^\ee \vv \ar\hh^{-\ee} \ ... \quad\agoesR1\quad ... \ar\hh^\ee\ \as\ii^\ee \, \vv \ar\hh^{-\ee}\ ... \, ,$$
and the point is that the hypotheses that $\rr$ commutes with~$\phi(\pi_\hh(\vv))$ and, on the other hand, with $\ss$, implies that it commutes with $\phi(\pi_\hh(\as\ii^\ee \vv))$ in every case. So the possible $\hh$-pair of~$\ww'$ is commutative, and $\ww'$ is regular

\ITEM2 Again, if $\hh$ is larger than~$\ii$ and~$\jj$, removing the letters of index~$\hh$ changes nothing essential and we obtain $\pi_\hh(\ww) \agoesR1 \pi_\hh(\ww')$. On the other hand, if $\hh = \ii > \jj$ holds, removing the letters of index~$\hh$ cancels the effect of the transformation, and we simply have $\pi_\hh(\ww) = \pi_\hh(\ww')$, hence (trivially) $\pi_\hh(\ww)\agozotiR \pi_\hh(\ww')$.

We now consider type~$2$ and assume that $\ww'$ is obtained from~$\ww$ by replacing $\as\ii^\ee \at\jj^{-\ee}$ with $\at\jj^{-\ee} \as\ii^\ee$, where $\ss \tt = \tt \ss$ is a relation of~$\RR$ and $\ee$ is~$\pm1$. Both for~\ITEM1 and~\ITEM2, the argument is similar to that of type~$1$, the only difference being that the commuted letters have now opposite signs (a specificity of the right-angled case).

Finally we consider type~$\infty$, that is, assume that $\ww'$ is obtained from~$\ww$ by inserting some pair $\as\ii^\ee \as\ii^{-\ee}$, $\ee = \pm1$, where $\ii$ is strictly larger than every index of letter occurring in~$\ww$. 

\ITEM1 If $\hh$ is not~$\ii$, in which case $\hh < \ii$ holds by definition, the possible $\hh$-pair of~$\ww$ provides an $\hh$-pair in~$\ww'$. We have no commutation hypothesis involving the letter~$\ss$, but, even if the inserted pair is nested in the $\hh$-pair as in
$$ ... \ \ar\hh^\ee \vv\vv' \ar\hh^{-\ee} \ ... \quad \agoes\infty \quad ... \ \ar\hh^\ee \vv \as\ii^\ee \as\ii^{-\ee} \vv' \ar\hh^{-\ee} \ ... \, ,$$
we are sure that $\rr$ commutes with every letter of~$\phi(\pi_\hh(\vv \as\ii^\ee \as\ii^{-\ee} \vv'))$ as the latter word coincides with~$\phi(\pi_\hh(\vv \vv'))$. So $\ww'$ is regular.

\ITEM2 By definition, we have $\pi_\hh(\ww) = \pi_\hh(\ww')$, whence again $\pi_\hh(\ww)\agozotiR \pi_\hh(\ww')$.

So the proof of~\ITEM1 and~\ITEM2 is complete, and it remains to deduce~\ITEM3. First, we claim that $\pi_\hh(\ww) \agozotiR \pi_\hh(\ww')$ holds for every~$\hh \ge 1$. Indeed, the result is obvious if $\hh$ is larger than every index of letter occurring in~$\ww$ and~$\ww'$ since, in this case, we have $\pi_\hh(\ww) = \ww$ and~$\pi_\hh(\ww') = \ww'$. Then the result follows from~\ITEM2, using a decreasing induction on~$\hh$ since applying~$\pi_\hh$ to a word where the largest index is~$\hh$ yields a word where the largest index is at most~$\hh-1$. So we obtain $\pi_1(\ww) \agozotiR \pi_1(\ww')$. Now, the word~$\pi_1(\ww')$ contains no letter of positive index. As no type~$\infty$ transformation can lead to a word containing no letter with positive index, the only possibility is that $\pi_1(\ww)\agozotR \pi_1(\ww')$ holds, as expected. 
\end{proof}

We can now combine the results to obtain the expected new proof of Proposition~\ref{P:RA}.

\begin{proof}[Proof of Proposition~\ref{P:RA}]
Let~$\ww$ be a word of~$\WWW\SS$ that represents~$1$ in~$\GR\SS\RR$. Then, by~\eqref{E:Equiv}, we have $\ww \gozoiR \ew$. Fix a special sequence $\ww_0, \ww_1, ..., \ww_\mm$ going from~$\ww$ to the empty word.

By Lemma~\ref{L:Lift}, we can find a special sequence of augmented words~$\tw_0, \tw_1, ...$ such that $\tw_0$ coincides with~$\ww_0$ and, for every~$\kk$, the word~$\ww_\kk$ is the image of~$\tw_\kk$ under~$\phi$.

By definition, the (augmented) word~$\tw_0$ contains no letter of positive index, hence it is regular. Hence, by Lemma~\ref{L:BasicStep}\ITEM1, every augmented word~$\tw_\kk$ is regular. 
Then Lemma~\ref{L:BasicStep}\ITEM3 implies $\pi_1(\tw_{\kk-1}) \agozotR \pi_1(\tw_\kk)$ for each~$\kk$. We deduce $\pi_1(\tw_0) \gozotR \pi_1(\tw_\mm)$, which is the expected relation $\ww \gozotR \ew$.
\end{proof}

\begin{exam}
\label{X:Abelian3}
Applying the above argument to the case of Example~\ref{X:Abelian2} amounts to collapsing all letters with positive index and forgetting indices. Starting from the initial sequence that contains a type~$\infty$ step (top), we lift it into an augmented sequence (middle, we drop the zero indices), and project it into a new sequence that contains no type~$\infty$ steps (down):

\begin{picture}(110,25)(3,-1)
\put(0,20){$\tta \ttB \ttc \ttA \ttb \ttC$}
\put(12.5,20){$\goes\infty$}
\put(18.5,20){$\tta \ttB \ttA \tta \ttc \ttA \ttb \ttC$}
\put(36.5,20){$\goesR1$} 
\put(43.5,20){$\tta \ttA \ttB \tta \ttc \ttA \ttb \ttC$}
\put(61,20){$\goes0$} 
\put(66,20){$\ttB \tta \ttc \ttA \ttb \ttC$}
\put(78.5,20){$\goesR1$} 
\put(84,20){$\ttB \ttc \tta \ttA \ttb \ttC$}
\put(96.5,20){$\goes0$} 
\put(101,20){$\ttB \ttc \ttb \ttC$}
\put(109,20){$\goesR1$} 
\put(114.5,20){$\ttB \ttb \ttc \ttC$}
\put(123,20){$\goes0$} 
\put(127,20){$\ttc \ttC$}
\put(132,20){$\goes0$} 
\put(137,20){$\ew$}

\put(0,10){$\tta\ttB \ttc \ttA \ttb \ttC$}
\put(12.5,10){$\agoes\infty$}
\put(17,10){$\tta \ttB \attA1 \atta1 \ttc \ttA \ttb \ttC$}
\put(36.5,10){$\agoesR1$} 
\put(42,10){$\tta \attA1 \ttB \atta1 \ttc \ttA \ttb \ttC$}
\put(61,10){$\agoes0$} 
\put(66,10){$\ttB \tta \ttc \ttA \ttb \ttC$}
\put(78.5,10){$\agoesR1$} 
\put(84,10){$\ttB \ttc \tta \ttA \ttb \ttC$}
\put(96.5,10){$\agoes0$} 
\put(101,10){$\ttB \ttc \ttb \ttC$}
\put(109,10){$\agoesR1$} 
\put(114.5,10){$\ttB \ttb \ttc \ttC$}
\put(123,10){$\agoes0$} 
\put(127,10){$\ttc \ttC$}
\put(132,10){$\agoes0$} 
\put(137,10){$\ew$}

\put(0,0){$\tta\ttB \ttc \ttA \ttb \ttC$}
\put(14,0){$=$}
\put(20,0){$\tta\ttB \ttc \ttA \ttb \ttC$}
\put(37.5,0){$=$} 
\put(45,0){$\tta \ttB \ttc \ttA \ttb \ttC$}
\put(60,0){$\goesR2$} 
\put(66,0){$\ttB \tta \ttc \ttA \ttb \ttC$}
\put(78.5,0){$\goesR1$} 
\put(84,0){$\ttB \ttc \tta \ttA \ttb \ttC$}
\put(96.5,0){$\goes0$} 
\put(101,0){$\ttB \ttc \ttb \ttC$}
\put(109,0){$\goesR1$} 
\put(114.5,0){$\ttB \ttb \ttc \ttC$}
\put(123,0){$\goes0$} 
\put(127,0){$\ttc \ttC$}
\put(132,0){$\goes0$} 
\put(137,0){$\ew$.}

\put(4.5,15){$\uparrow\!\phi$}
\put(4.5,5){$\downarrow\!\pi_1$}
\put(24.5,15){$\uparrow\!\phi$}
\put(24.5,5){$\downarrow\!\pi_1$}
\put(49.5,15){$\uparrow\!\phi$}
\put(49.5,5){$\downarrow\!\pi_1$}
\put(70.5,15){$\uparrow\!\phi$}
\put(70.5,5){$\downarrow\!\pi_1$}
\put(88.5,15){$\uparrow\!\phi$}
\put(88.5,5){$\downarrow\!\pi_1$}
\put(103.5,15){$\uparrow\!\phi$}
\put(103.5,5){$\downarrow\!\pi_1$}
\put(117.5,15){$\uparrow\!\phi$}
\put(117.5,5){$\downarrow\!\pi_1$}
\put(127.5,15){$\uparrow\!\phi$}
\put(127.5,5){$\downarrow\!\pi_1$}
\put(136.5,15){$\uparrow$}
\put(136.5,5){$\downarrow$}
\end{picture}

\noindent (Of course, this example is nearly trivial as one type~$\infty$ transformation only occurs in the initial sequence.) 
\end{exam}

%%%%

\subsection{Further remarks}
\label{SS:Further}

Extending the method of Subsection~\ref{SS:RAAG} to more general Artin--Tits presentations is a challenge we leave open. The technical problem is to find augmented versions of type~$2$ transformations: in the right-angled case, the latter just are commutations, and the definition is clear; in more complicated cases, things are not so clear. Let us just mention that, in the case of a length~$3$ relation $\ss \tt \ss = \tt \ss \tt$, geometric interpretations suggest that the proper way of defining an augmented version of the type~$2$ transformation $\ss\inv \tt \goesR{2} \tt \ss \tt\inv \ss\inv $ could be to put
$\as\ii\inv \at\jj \agoesR{2} \at\kk \as\jj \at\ii\inv \as\kk\inv $, where $\kk$ is larger than every index of letter occurring in the current word. 

Let us still mention two more approaches. First, a natural way for proving Conjecture~\ref{C:Main} could be to establish the confluence of the relation~$\gozotR$, namely to prove that, if we have both $\ww \gozotR \ww'$ and $\ww \gozotR \ww''$, then there exists a word~$\ww'''$ satisfying $\ww' \gozotR \ww'''$ and $\ww'' \gozotR \ww'''$. However, as the relation~$\gozotR$ need not be Noetherian in general (that is, there may exist infinite special sequences), establishing confluence is not easy.

On the other hand, working with $\gozotR$ is in general difficult because this relation is highly non-deterministic: for every word~$\ww$, there exist in general many words~$\ww'$ such that $\ww \gozotR \ww'$ is satisfied. In~\cite{Dfo}, an algorithm called handle reduction is constructed in the case of classical braids, that is, of Artin--Tits groups of type~A. This algorithm, called \emph{handle reduction}, extends free group reduction, and it returns for every word~$\ww$ an equivalent word~$\HR\ww$ with the property that, if $\ww$ represents~$1$, then $\HR\ww$ is the empty word. However, this algorithm is not a normal form, as $\ww \eqR \ww'$ need not imply $\HR\ww = \HR{\ww'}$. A precise analysis of the algorithm shows that $\ww \gozotR \HR\ww$ always holds and, therefore, we obtain another proof of Property~$\Prop$ in this case. In this approach, handle reduction is used as a strategy for choosing, for every word~$\ww$, a distinguished word among the many words~$\ww'$ that satisfy $\ww \gozotR \ww'$. It is natural to wonder whether such an approach, which a priori is less ambitious than constructing a normal form, might apply to more general situations. 

%%%%%%%%

\end{document}